\documentclass[12pt, reqno]{amsart}

\usepackage[frenchb]{babel}
\usepackage[utf8]{inputenc}
\usepackage[T1]{fontenc}
\usepackage{tikz,tikz-3dplot}
\usepackage{amsmath,amsfonts,amssymb,amsmath,latexsym,mathrsfs}

\usepackage[all]{xy}

\usepackage{url}
\usepackage{tikz}
\newtheorem{theorem}{Théorème}[section]

\newtheorem{proposition}[theorem]{Proposition}

\newtheorem{definition}[theorem]{Définition}

\newtheorem{corollary}[theorem]{Corollaire}
\newtheorem{prop-def}[theorem]{Définition-Proposition}
\newtheorem{remarque}[theorem]{Remarque}

\newtheorem{lemma}[theorem]{Lemme}

\newcommand{\Q}{ {\mathbb Q} }
\newcommand{\R}{ {\mathbb R} }

\newcommand{\C}{ {\mathbb C} }
\newcommand{\Z}{ {\mathbb Z} }
\newcommand{\N}{ {\mathbb N} }
\newcommand{\PP}{ {\mathbb P} }
\newcommand{\K}{ {\Bbbk} }
\title{Sur certains espaces de configurations associés aux sous-groupes finis de $\mathrm{PSL}_2(\C) $}
\author{Mohamad MAASSARANI}
\address{IRMA, Université de Strasbourg, 7 rue René Descartes, 67084
Strasbourg, France}
\email{maassarani@math.unistra.fr}
\begin{document}
\maketitle
\begin{abstract}
On étudie des espaces de configurations $\mathrm{Cf}_G(n,\PP^1_*)$ liés à l'action d'un groupe fini d'homographies $G$ de $\PP^1$ $(n\in \N^{*})$. On construit une connexion plate sur cet espace à valeurs dans une algèbre de Lie $\hat{\mathfrak{p}}_n(G) $. On établit un isomorphisme d'algèbres de Lie filtrées entre $\hat{\mathfrak{p}}_n(G)$, l'algèbre de Lie de Malcev du groupe fondamental de cet espace et le complété pour le degré du gradué associé à cette algèbre de Lie. Ceci est obtenu grâce à la représentation de monodromie d'une connexion et une étude du groupe fondamental. 
\end{abstract}

\section*{introduction}
L'un des invariants associés à un espace topologique $X$ en homotopie rationelle est son modèle minimal. Le calcul du modèle minimal de $X$, plus précisement du 1-modèle minimal, permet d'obtenir l'algèbre de Lie de Malcev de $\pi_1(X)$, le groupe fondamental de $X$, par un processus de dualisation. Dans \cite{FM}, Fulton et MacPherson calculent explicitement des modèles des espaces de configurations $\mathrm{Cf}_n(X)=\{(p_1,\cdots,p_n)\in X^n\vert p_i\neq p_j \text{ si } i\neq j \}$, pour $X$ une variété projective complexe lisse. Ces modèles sont ensuite simplifiés dans \cite{Kriz}, puis utilisés par Bezrukavnikov (\cite{bezr}) qui obtient une présentation de l'algèbre de Lie $\mathrm{Lie}(\pi_1(\mathrm{Cf}_n(S))$ de Malcev de $\pi_1(\mathrm{Cf}_n(S))$ pour $S$ une surface de genre supérieur à un.\\
Une approche alternative, motivée par \cite{Dr}, repose sur l'utilisation de connexions plates et d'informations sur le groupe fondamental. En utilisant cette approche, différents résultats sont obtenus :
\begin{enumerate}
\item calcul de l'algèbre de Lie de Malcev de $\mathrm{Cf}_n(S)$ pour $S$ de genre $g(S)=1$ (\cite{CEE}) puis en genre $g(S)>1$ (\cite{BE3}) ; ce qui donne une autre démonstration aux présentations obtenues par Bezrukavnikov.
\item  calcul de l'algèbre de Lie de Malcev d'"espaces de 
configurations d'orbites", au sens de \cite{CKX}, pour les groupes des racines de l'unité opérant sur $\C^\times$ (\cite{BE}).
\end{enumerate}
Dans ce papier, on
considère plus généralement $G$ un groupe fini d'homographies agissant sur $\PP^1$ et l'espace associé :
$$\mathrm{Cf}_G(n,\PP^1_*)=\{(p_1,\cdots,p_n) \in (\PP^1_*)^n \vert p_i \neq g \cdot p_j ; \text{ pour } i\neq j \text{ et } g\in G \},$$
dans lequel $\PP^1_*$ est l'ensemble des points de $\PP^1$ à stabilisateur trivial pour $G$.
En utilisant la méthode des connexions plates, on calcule une présentation de l'algèbre de Lie de Malcev de $\pi_1(\mathrm{Cf}_G(n,\PP^1_*))$ et on montre (théorème \ref{Th}) que cette algèbre de Lie est isomorphe à la complétion pour le degré de son gradué associé qui coïncide avec une algèbre de Lie explicite $\hat{\mathfrak{p}}_n(G)$ (définition \ref{def AL}). On obtient par ailleurs la 1-formalité de $\mathrm{Cf}_G(n,\PP^1_*)$.\\\\
Détaillons les étapes permettant d'obtenir ce résultat. Dans la première section, on définit une algèbre de Lie $\mathfrak{p}_n(G)$, puis on construit une connexion plate sur $\mathrm{Cf}_G(n,\PP^1_*)$ à valeurs dans $\mathfrak{p}_n(G)$. Cette connexion nous donne une représentation de monodromie $\rho_{\tilde{\underline{q}}_n}:\pi_1(\mathrm{Cf}_G(n,\PP^1_*))\longrightarrow \mathcal{G}(\mathrm{U}\mathfrak{p}_n(G)(\C))$, où $\mathcal{G}$ est le foncteur qui à une algèbre de Hopf associe le groupe de ses éléments diagonaux. \\
On rappelle en section 2 quelques notions de topologie différentielle qui seront utilisées dans la section 3, laquelle est consacrée à l'étude du groupe fondamental d'un espace de configurations d'orbites associé à une surface munie d'une action d'un groupe fini. Dans cette section, on donne notamment une famille génératrice de $\Gamma_n:=\pi_1(\mathrm{Cf}_G(n,\PP^1_*))$ et des relations entre ces éléments de $\Gamma_n$.\\
La quatrième section est consacrée à des rappels de notions liées aux algèbres de Lie de Malcev et aux algèbres de Hopf complètes.\\
Dans la section 5, on utilise le morphisme de monodromie de la section 1 pour construire un morphisme $\mathrm{Lie}(\rho)$ de l'algèbre de Lie de Malcev $\mathrm{Lie}(\Gamma_n(\C))$ de $\Gamma_n$ sur $\C$ dans $\widehat{\mathfrak{p}}_n(G)(\C)$. D'autre part, on obtient grâce aux générateurs et relations de $\Gamma_n$ un morphisme $\phi_\C : \widehat{\mathfrak{p}}_n(G)(\C) \to \widehat{\mathrm{gr}}\mathrm{Lie}(\Gamma_n(\C)) $, où l'espace d'arrivée est le complété pour le degré du gradué associé de $\mathrm{Lie}(\Gamma_n(\C))$. En examinant la composée de $\mathrm{Lie}(\rho)$ avec $ \phi_\C$ , on conclut que les trois algèbres de Lie $\mathrm{Lie}(\Gamma_n(\C))$, $\widehat{\mathrm{gr}}\mathrm{Lie}(\Gamma_n(\C))$ et $\widehat{\mathfrak{p}}_n(G)(\C)$ sont isomorphes en tant qu'algèbres de Lie filtrées.\\
Enfin, la dernière section, on construit des torseurs dont la composée $ \phi_\C \circ \mathrm{Lie}(\rho)$  de la section 5 est un point complexe. Ensuite, on utilise un résultat sur l'existence de points rationnels de ces torseurs pour déduire que $\mathrm{Lie}(\Gamma_n(\Q))$, $\widehat{\mathrm{gr}}\mathrm{Lie}(\Gamma_n(\Q))$ et $\widehat{\mathfrak{p}}_n(G)(\Q)$ sont isomorphes comme algèbres de Lie filtrées. \\
Notons que la 1-formalité des espaces $\mathrm{Cf}_G(n,\PP^1_*)$ est également une conséquence du résultat principal de \cite{Koh}, et dans le cas ou $G$ est un groupe de racines de l'unité, une présentation de l'algèbre d'holonomie peut également être déduite de ce résultat.

\section{ Connexion sur l'espace de configuration  $\mathrm{Cf}_G(n,\PP^1_*)$ et représentation de monodromie.}
Dans cette section, on considère une action d'un groupe fini $G$ sur $\PP^1$ (sect. 1.1). On lui associe un espace de configuration $\mathrm{Cf}_G(n,\PP^1_*)$ (sect. 1.4) et une algèbre de Lie $\mathfrak{p}_n(G)$ (sect. 1.2). Après des rappels sur les connexions formelles (sect. 1.3), on définit une telle structure sur $\mathrm{Cf}_G(n,\PP^1_*)$ associée à l'algèbre de Lie $\mathfrak{p}_n(G)$  (sect. 1.4) et on montre sa platitude (sect. 1.5). On calcule alors les termes de bas degré de la représentation de monodromie associée (sect. 1.6). 
\subsection{Le groupe $G$ opérant sur $\PP^1$.}

\subsubsection{Action de $G$ sur $\PP^1$}\label{sec action} On a la suite de morphismes de groupes suivante :

\[
\xymatrix{
     \rm{SO}_3(\R) \simeq \rm{PSU}_2(\C)\ar@{^{(}->}[r]   & \mathrm{PSL}_2(\C)  \\
      }
  \]
  Par ailleurs, on a une action $ \rm{PSL}_2(\C)\to \rm{Aut}(\PP^1)$ par homographies et une action $\rm{SO}_3(\R) \to \rm{Aut}(S^2)$ par rotations. Enfin, il existe un identification $\PP^1 \simeq S^2$ compatible aux actions. L'action de $\rm{SO}_3(\R)$ sur $S^2$ commute à l'antipode. De façon analogue, l'action de $\rm{PSU}_2(\C)$ sur $\PP^1$ commute à l'involution $z\mapsto \mathrm{at}(z):=\frac{-1}{\bar{z}}$. Dans la suite, on fixe un sous groupe fini $G$ de $\rm{SO}_3(\R) \simeq \rm{PSU}_2(\C)$.
On sait que $G$ est soit cyclique ou diédral, soit isomorphe à un des groupes d'isométries des solides platoniciens $\mathfrak{A}_4$, $\mathfrak{S}_4$, $\mathfrak{A}_5$.
\subsubsection{Points fixes et stabilisateurs}
On note $\PP^1_*$ l'ensemble des points de $\PP^1$ à stabilisateur trivial pour $G$.
\begin{proposition}
\label{lem orbite} Pour tout $g\in G$ et $p\in \PP^1$, on pose $\mathrm{Fix}(g)=\{q\in \PP^1\: \vert \: g\cdot q=q \:\}$ et on note $\mathrm{stab}(p)$ le stabilisateur de $p$ pour l'action de $G$ sur $\PP^1$. Alors :

\begin{enumerate}
\baselineskip 18pt 
\item L'application $\mathrm{at}$ se restreint en une involution de $\PP^1\setminus \PP^1_*$. Pour tout $g\neq 1$, $\mathrm{Fix}(g)$ est de la forme $\{p,\mathrm{at}(p)\}$ avec $p\neq \mathrm{at}(p)$.  
\item Il existe un sous-ensemble fini $\mathfrak{p}$ de $\PP^1\setminus \PP^1_* $ satisfaisant $\PP^1\setminus \PP^1_*=\mathfrak{p}\sqcup \mathrm{at}(\mathfrak{p})$ et $G\setminus \{1\}= \underset{p \in \mathfrak{p}}{\sqcup} (\mathrm{stab}(p)\setminus \{1\})$.
\end{enumerate}
\end{proposition}
\begin{proof}
Il suffit de montrer la proposition pour $G$ un groupe fini de rotations de la sphère. Dans ce cadre, l'application $\mathrm{at}$ n'est autre que l'antipode de $S^2$. Ce qui montre (1). L'existence d'un $\mathfrak{p}$ fini satisfaisant $\PP^1\setminus \PP^1_*=\mathfrak{p}\sqcup \mathrm{at}(\mathfrak{p})$ est immédiate à partir de (1). Un tel ensemble satisfait automatiquement la dernière condition de (2). En effet, si l'intersection $\mathrm{stab}(p)\cap \mathrm{stab}(q)$ pour $p\neq q$ est différente de $\{1\}$, alors (1) nous mène à la contradiction $q\in\{p,\mathrm{at}(p)\}$. Ce qui montre la proposition. 
\end{proof}
\subsection{L'algèbre de Lie $\mathfrak{p}_n(G)$ } Soit $n$ un entier strictement positif et $\K$ un corps. On note $\mathcal{O}(p)$ l'orbite pour $G$ d'un point $p \in \PP^1$. 
\begin{definition}\label{def AL}
 On définit $\mathfrak{p}_n(G)(\K)$ comme la $\K$-algèbre de Lie engendrée par les éléments $ X_{ij}(g)$, pour $i\neq j \in [1,n]$ et $g \in G$, les $X_i(q)$ pour $i\in [1,n]$ et $q\in \PP^1\setminus \PP^1_*$ , soumis aux relations :  
\begin{equation} X_{ij}(g)=X_{ji}(g^{-1}) ,\text{ pour $i,j \in [1,n]$ distincts et $g\in G$,}\end{equation}

\begin{equation}\underset{q\in \PP^1\setminus \PP^1_*}{\sum} X_i(q) +\underset{ \underset{m\neq i}{m\in[1,n]} }{\sum} \: \underset{g\in G}{\sum} X_{im}(g)=0, \text{ pour $i\in [1,n]$}\end{equation} ,  
\begin{equation} [X_{ij}(g),X_{kl}(g')]=0, \text{ pour $i,j,k,l\in [1,n]$ distincts et $g,g'\in G$,}\end{equation} 

\begin{align} [X_{ij}(g),X_{kj}(g'g)+X_{ki}(g')]=[X_{i}(p),X_{jk}(g')]=0, &\text{ pour $i,j,k \in[1,n]$ distincts, } \endline 
&\text{$p\in  \PP^1\setminus\PP^1_*$ et $g,g'\in G$,}\end{align} 
\begin{equation} [X_{i}(p),X_{j}(q)]=0,\end{equation} 
\begin{equation}[X_{ij}(g),X_j(p) +X_i(g\cdot p) +\: \underset{h\in \mathrm{stab}(p)}{\sum} X_{ij}(gh)]=0,\end{equation} 
\begin{equation} [X_{j}(p),X_i(g\cdot p) +\: \underset{h\in \mathrm{stab}(p)}{\sum} X_{ij}(gh)]=0,\end{equation} 

pour $i,j\in [1,n]$ distincts, $g\in G$, $p \in \PP^1\setminus \PP^1_*$ et $q\in (\PP^1\setminus \PP^1_*)\setminus \mathcal{O}(p)$.
\end{definition}
L'algèbre $\mathfrak{p}_n(G)(\K)$ est munie d'une graduation pour laquelle chaque générateur $X_{ij}(g)$ est de degré 1. On a :
 \[ \mathfrak{p}_n(G)(\K)=\underset{k>0}{\bigoplus} \: \mathfrak{p}_n^k(G)(\K),\] où $\mathfrak{p}_n^k(G)(\K)$ est la composante homogène de degré $k$. On note $\hat{\mathfrak{p}}_n(G)(\K)$ la complétion de $\mathfrak{p}_n(G)(\K)$ pour le degré. \\
D'autre part, l'algèbre enveloppante $\mathrm{U}\mathfrak{p}_n(G)(\K)$ de $\mathfrak{p}_n(G)(\K)$,  hérite de $\mathfrak{p}_n(G)(\K)$ une structure d'algèbre graduée pour le degré. On note $\widehat{\mathrm{U}}\mathfrak{p}_n(G)(\K)$ la complétion de cette algèbre enveloppante pour le degré. L'algèbre $\mathfrak{p}_n(G)(\K)$ étant engendrée en degré un, la complétion de $\mathrm{U}\mathfrak{p}_n(G)(\K)$ pour le degré et la complétion pour les puissances de l'idéal d'augmentation coïncident. Enfin, $\widehat{\mathrm{U}}\mathfrak{p}_n(G)(\K)$ est une algèbre de Hopf complète.\\
Dans la suite on omettra dans les notations $G$ ou $\K$, si le contexte est clair.
\begin{remarque}Le groupe symétrique $\mathfrak{S}_n$ et le groupe $G^n$ agissent sur l'algèbre de Lie $\mathfrak{p}_n(G)$. Ces actions sont définies par :
$$\underline{g}\cdot X_{ij}(h)=X_{ij}(g_{i}hg_j^{-1}),\qquad \underline{g} \cdot X_{i}(q)=X_{i}(g_{i}\cdot q),$$
$$\sigma \cdot X_{ij}(h)=X_{\sigma(i)\sigma(j)}(h),\qquad \sigma \cdot X_{i}(q)=X_{\sigma(i)}(q),$$
pour $\underline{g}=(g_1,\cdots,g_n) \in G^n$ et $\sigma \in \mathfrak{S}_n$.  
\end{remarque}
\subsection{Connexions formelles}\label{sec formelle}
Dans cette sous-section, on passe en revue certaines notions sur les connexions formelles, leur platitude et les représentations de monodromie induites. Les résultats étant bien connus, on les donnera sans démonstration.\\\\
Soit $X$ une variété analytique complexe et  $A$  une $ \mathbb{C}$-algèbre complète unitaire graduée : $A={\prod}_{k\geq 0} A_k$ telle que $A_0=\mathbb{C}$, $A_k \cdot A_l \subset A_{k+l}$ pour $k,l \in \mathbb{N}$ et que les composantes homogènes soient de dimension finie. L'algèbre $A$ est supposée munie de la topologie produit. On note $\Omega^{\bullet}(X)$ l'algèbre des formes holomorphes sur $X$ et $\Omega^{\bullet}(X) \widehat{\otimes} A $ la complétion de l'algèbre $\Omega^{\bullet}(X) \otimes_\C A$ pour la filtration $\{ \Omega^{\bullet}(X) \otimes_\C A_{\geq k} \}_{k\geq0}$ (on note $A_k=\prod_{n\geq k}A_n$). On notera $\wedge$ le produit de $\Omega^{\bullet}(X) \widehat{\otimes} A $.
 On se donne aussi une 1-forme holomorphe sur $X$ à valeurs dans $A_{\geq 1}$ (i.e un élément de $\Omega^{1}(X) \widehat{\otimes} A $), qu'on notera $\omega$.
\begin{definition}\label{def formelle}
 Le triplet $(X,A,\omega)$ comme ci-dessus est appelé connexion formelle sur $X$, à valeurs dans $A$. Cette connexion est dite plate si $(d\widehat{\otimes} \mathrm{id}_A)(\omega)-\omega \wedge \omega=0$ 
\end{definition}

\begin{theorem}\label{exi solution}
Soit $(V,A,\Omega)$ une connexion formelle et $v$ un point de $V$. Si cette connexion est plate et $V$ est simplement connexe, alors l'équation :
\[ (d\widehat{\otimes} \mathrm{id}_A) F=\Omega\wedge F ,\]
d'inconnue $F \in \Omega^0(V) \widehat{\otimes} A$, admet une unique solution dans $ \Omega^0(V) \widehat{\otimes} A$, à condition initiale fixée $(F(v)=f \in A)$. Si $f$ est inversible dans $A$ $(f\in A^\times)$ , alors $F$ est à valeurs dans $A^\times$.
\end{theorem}
Dans la suite de cette sous-section, on se donne une connexion plate  ($X$, $A$, $\omega$) comme dans la définition 1.4. On note $r : \tilde{X} \longrightarrow X$ le revêtement universel de $X$.\\
La connexion ($\tilde{X}$, $A$, $r^*(\omega)$) est une connexion plate sur un espace simplement connexe. On est dans le cadre du théorème 1.5. Etant donné un point $x_0$ de $\tilde{X}$, l'équation $dF=r^*(\omega) \wedge F$, $F(x_0)=1$ admet une unique solution qu'on notera $F(x,x_0)$.\\\\
Ainsi, on définit une application $\rho_{x_0} : \pi_1(X,r(x_0)) \longrightarrow 1+ A_{\geq 1}$ qui à $\gamma \in \pi_1(X,r(x_0)) $ associe $F(\gamma \cdot x_0 ,x_0)$, où $\gamma$ agit via le morphisme de monodropmie du revêtement.
\begin{proposition}
Soit $\gamma ,\: \gamma_1,\:\gamma_2  \in \pi_1(X,r(x_0)) $. On fait agir $\pi_1(X,r(x_0))$ à gauche sur $\tilde{X}$, via le morphisme de monodromie du revêtement. On a les égalités suivantes :
\begin{enumerate}
\baselineskip 18pt 
\item $F(\gamma \cdot x,  \gamma \cdot y)=F(x,y)$.
\item $F(x,y) F(y,z)=F(x,z) $.
\item $F((\gamma_1 \gamma_2)\cdot x_0, x_0)=F(\gamma_2 \cdot x_0 ,x_0) F(\gamma_1 \cdot x_0, x_0)$.
\end{enumerate} 
\end{proposition}
La démonstration de (3) repose sur les points (1) et (2).\\
Cette proposition nous permet d'affirmer que si la connexion considérée est plate alors l'application $\rho_{x_0}$ associée induit un antimorphisme de groupe de $\pi_1(X, r(x_0))$ dans $1+A_{\geq 1}$. On dira que $\rho_{x_0}$  est la représentation de monodromie associée à la connexion en question en $x_0$.\\\\ 
On termine cette sous-section par un résultat dans le cas où $A$ est une algèbre de Hopf complète de coproduit noté $\Delta$. On note $\mathcal{P}(A)$ l'ensemble des primitifs de $A$ et $\mathcal{G}(A)=\{ a \in A \vert \Delta(a)=a \widehat{\otimes} a,\: a \in A^{\times}\}$ l'ensemble des éléments diagonaux.
\begin{theorem} 
Soit $(V, A ,\omega)$ une connexion formelle, où $V$ est simplement connexe. Supposons de plus que $A$ est une algèbre de Hopf complète, que $\omega$ est à valeurs dans $\mathcal{P}(A)$ et que la condition initiale de l'équation $(d\widehat{\otimes} \mathrm{id}_A) F=\Omega\wedge F$ est dans $\mathcal{G}(A)$. Alors la solution $F$ du théorème \ref{exi solution} est à valeurs dans $\mathcal{G}(A)$.
\end{theorem}
Ainsi, si $A$ est une algèbre de Hopf complète et $\omega$ est plate et à valeurs dans $\mathcal{P}(A)$, alors la représentation de monodromie associée est à valeurs dans $\mathcal{G}(A)$.  

\subsection{Connexion sur l'espace de configuration $\mathrm{Cf}_G(n,\PP^1_*)$}
On a une action  de $G$ sur $\PP^1$. En considérant cette action on définit l'espace de configurations $\mathrm{Cf}_G(n,\PP^1_*)=\{(p_1,\cdots,p_n) \in (\PP^1_*)^n \vert p_i \neq g p_j ; i\neq j, g\in G \}$.
On va construire une connexion formelle sur l'espace $\mathrm{Cf}_G(n,\PP^1_*)$, à valeurs dans l'algèbre $\widehat{\mathrm{U}}\mathfrak{p}_n(G)(\mathbb{C})$.
\begin{prop-def}\label{def forme}
Il existe une unique 1-forme holomorphe sur $\mathrm{Cf}_G(n,\PP^1_*)$ à valeurs dans $\widehat{\mathrm{U}}\mathfrak{p}_n(G)(\mathbb{C}) $, qu'on notera $\omega$, dont la restriction à $U:=(\PP^1 \setminus \{\infty \})^n\cap \mathrm{Cf}_G(n,\PP^1_*) $ est donnée par : 
\begin{equation}\label{forme 123}
\omega_{\vert U}=\underset{i \in[1,n]}{\sum}\:\: \underset{p\in \PP^1\setminus \PP^1_*}{\sum} d \mathrm{log} (z_i-p)\widehat{\otimes}  X_{i}(p)+\underset{i \neq j \in[1,n]}{\sum}\: \: \underset{g\in G}{\sum} \: \:d_{z_i}\mathrm{log}(z_i-g\cdot z_j) \widehat{\otimes}  X_{ij}(g),\end{equation}
où l'on pose $d\mathrm{log}(z-\infty)=0$.
\end{prop-def}
\begin{proof}
Soit $q \in \PP^1\setminus (\PP^1_*\cup \{\infty\})$. En utilisant la relation (2) de la définition \ref{def AL}, on vérifie que :
\[ \omega_{\vert U}=\underset{i \in[1,n]}{\sum}\:\: \underset{p\in \PP^1\setminus \PP^1_*}{\sum} d \mathrm{log} (\frac{z_i-p}{z_i-q})\widehat{\otimes}  X_{i}(p)+\underset{i \neq j \in[1,n]}{\sum}\: \: \underset{g\in G}{\sum} \: \:d_{z_i}\mathrm{log}(\frac{z_i-g\cdot z_j}{z_i-q}) \widehat{\otimes}  X_{ij}(g).\]
Or, toutes les 1-formes apparaissant dans cette écriture sont holomorphes sur $\mathrm{Cf}_G(n,\PP^1_*)$.  
\end{proof}
Notons que $\omega$ est à valeurs dans l'ensemble des primitifs de $\widehat{\mathrm{U}}\mathfrak{p}_n(G)(\mathbb{C}) $. Enfin, le triplet ($\mathrm{Cf}_G(n,\PP^1_*)$, $\widehat{\mathrm{U}}\mathfrak{p}_n(G)(\mathbb{C})$ , $\omega$) est une connexion formelle, au sens de la sous-section \ref{sec formelle}.
\subsection{Platitude de la connexion} 
Dans cette sous-section on montre que la connexion $\omega$ de la définition \ref{def forme} est plate. 
\begin{lemma}\label{rel forme}  
\begin{enumerate}
\baselineskip 18pt 
\item Soit $h$ un élément de $G$ , on a l'égalité suivante entre formes méromorphes sur $\C^3$  :
 \[\frac{dx\wedge dy}{(x-z)(y-h\cdot x  )}=\frac{dx\wedge dy}{(x- z)(y-h z)}-\frac{dx\wedge dy}{(x-h^{-1}\cdot y)(y-h\cdot z)}+\frac{dx}{x-h^{-1}\cdot y}\wedge \omega_h(y),\]
où $\omega_h(y)=\begin{cases}
 d_y \mathrm{log} (y-h\cdot \infty) & \text{si } h\cdot \infty \neq \infty \\ 
0 & \text{sinon}
\end{cases}.\\$
\item Soit $p\in \PP^1$ et supposons que $h \in \mathrm{stab}(p) \setminus \{1\}$. On a l'égalité suivante entre formes méromorphes sur $\C$ :
\[\frac{dz}{z-h\cdot z}+\frac{d z}{z-h^{-1} \cdot z}=\frac{dz}{z-p}+\frac{d z}{z-\mathrm{at}(p)},\]
où $\mathrm{at}$ est l'involution de la proposition \ref{lem orbite} et l'on pose $\frac{1}{z-\infty}=0$.
\end{enumerate}
\end{lemma}
\begin{proof}
Notons d'abord que si $S= \begin{pmatrix}a&b \\c&d  \end{pmatrix}  \in \mathrm{GL}_2(\C)$, alors :
\[S\cdot \alpha-S\cdot \beta=\frac{(\alpha-\beta)det(S)}{P_S(\alpha)P_S(\beta)}, \qquad \qquad (A_S)\]
où $S\cdot z= \frac{az+b}{cz+d}$ et $P_S(z)=cz+d$. Montrons la première égalité. Une décomposition en éléments simples, selon la variable $x$, donne : \[\frac{dx\wedge dy}{(x-z)(y-h\cdot x  )}=\frac{dx\wedge dy}{(x- z)(y-h\cdot z)}-\frac{dx\wedge d(h^{-1} \cdot y)}{(x-h^{-1}\cdot y)(h^{-1}\cdot y- z)}.\]
D'autre part, en utilisant $(A_S)$ pour $S=h^{-1},\alpha=y, \beta=h\cdot z$ et en appliquant $d_y\mathrm{log}$, on trouve : \[\frac{d(h^{-1}\cdot y)}{h^{-1}\cdot y-z}=d_y\mathrm{log}(h^{-1}\cdot y- z)=d_y\mathrm{log}( y-h\cdot z)- d_y\mathrm{log}(P_{h^{-1}}(y)).\]
Or, $d_y\mathrm{log}(P_{h^{-1}}(y))$ n'est autre que la forme $\omega_h(y)$ de l'énoncé, ce qui montre (1).\\\\
Montrons (2). On a vu (proposition \ref{lem orbite}) que $\vert \mathrm{Fix}(h)\vert =2$. Posons $h(z)=\frac{az+b}{cz+d}$ et $\mathrm{Fix}(h)=\{p,\mathrm{at}(p)\}$. On distinguera deux cas :  \\
Si $\infty \notin \mathrm{Fix}(h)$, on a $c\neq 0$ et $p ,\mathrm{at}(p)$ sont deux solutions distinctes de $cz^2+(d-a)z-b=0$. On en déduit que $p+\mathrm{at}(p)=\frac{a-d}{c}$ et que $p\cdot \mathrm{at}(p)=\frac{-b}{c}$. Par conséquent, on a :
\[ \frac{1}{z-p}+\frac{1}{z-\mathrm{at}(p)}= \frac{2z-(p+\mathrm{at}(p))}{(z-p)(z-\mathrm{at}(p))}=\frac{2cz+d-a}{cz^2+(d-a)z-b}.\]
En développant $\frac{1}{z-h\cdot z}+\frac{1}{z-h^{-1} \cdot z}$ avec $h(z)=\frac{az+b}{cz+d}$, on trouve (2) pour  $\infty \notin \mathrm{Fix}(h)$. \\
Si $\infty \in \mathrm{Fix}(h)$, $h$ est de la forme $h(z)=ez+f$ avec $e\neq 1$. Un calcul immediat donne l'égalité souhaité. Ce qui montre le lemme.
\end{proof}\
Montrons la platitude de la connexion :
\begin{proposition}
La connexion formelle $(\mathrm{Cf}_G(n,\PP^1_*),  \omega)$ est une connexion plate. La 1-forme $\omega$ vérifie : $(d\widehat{\otimes} \mathrm{id}_{\widehat{\mathrm{U}}\mathfrak{p}_n})\omega=\omega \wedge \omega =0$.
 \end{proposition}
\begin{proof} Il suffit de montrer ces égalités sur $U=(\PP^1 \setminus \{\infty \})^n\cap \mathrm{Cf}_G(n,\PP^1_*) $. Les calculs seront effectués dans $ \mathcal{M}(U) \widehat{\otimes} \widehat{\mathrm{U}}\mathfrak{p}_n(G)(\mathbb{C})$, où $ \mathcal{M}(U)$ est l'algèbre des formes méromorphes sur $U$. On note $X w$ l'élément $w\widehat{\otimes}X \in \mathcal{M}(U) \widehat{\otimes}\widehat{\mathrm{U}}\mathfrak{p}_n(G)(\mathbb{C})$. Pour simplifier la preuve, on pose $\mathbb{P}=\PP^1\setminus \PP^1_*$ et on utilisera la convention $\frac{1}{z-\infty}=0$.\\\\
La relation (1) de la définition \ref{def AL} donne : \[\omega=\underset{i \in[1,n]}{\sum}\:\: \underset{p\in \mathbb{P}}{\sum} \frac{X_{i}(p)dz_i}{z_i-p}  +\underset{i< j}{\sum} \: \underset{g \in G}{\sum} X_{ij}(g)[ d_{z_i}\mathrm{ log}(z_i-g \cdot z_j)+  d_{z_j}\mathrm{log}(z_j-g^{-1}\cdot z_i)].\]
En posant $g(z)=\frac{a_g z+b_g}{c_g z +d_g}$ et en simplifiant l'expression ci-dessus, on trouve : 
 \[\omega= \underset{i \in[1,n]}{\sum}\:\: \underset{p\in \mathbb{P}}{\sum} \frac{X_{i}(p)dz_i}{z_i-p}+ \underset{i< j}{\sum} \: \underset{g \in G}{\sum} X_{ij}(g) d \mathrm{log}(f_g^{ij}),\]
où $f_g^{ij}=z_i(c_g z_j +d_g)-(a_g z_j +b_g)$. Comme $d\circ d=0$ on en déduit que  $d\omega =0$, .\\\\
Passons à l'étude du terme $ \omega \wedge \omega$. On note $\omega^l$ la partie du membre de droite de (\ref{forme 123}) contenant $ dz_l$. Alors $\omega$ est donnée par :
\begin{equation}\label{decomp omega}
\omega= \overset{n}{\underset{l=1}{\sum}}\: \: \omega^l.
\end{equation}
En utilisant les relations $[X_{ij}(g),X_{kl}(h)]=[X_{i}(q),X_{kl}(h)]=0$  (voir définition \ref{def AL}), on trouve pour $i\neq k$:
\begin{equation}\label{decomp omega wedge}
 \omega^i\wedge \omega^k+\omega^k\wedge \omega^i=(A+B+C+D+ \underset{j\vert j \notin\{i,k\}}{\sum}(E_{j}^1+E_{j}^2+E_{j}^3)) dz_i\wedge dz_k,
 \end{equation}
où 

\[A=\underset{p\in \mathbb{P}}{\sum} \: \underset{q\in \mathcal{O}(p)}{\sum} \frac{[X_i(p),X_k(q)]}{(z_i-p)(z_k-q)},\qquad \qquad B=\underset{p\in \mathbb{P},g \in G}{\sum} \frac{[X_i(p),X_{ki}(g)]}{(z_i-p)(z_k-g\cdot z_i)} \]

\[C=\underset{p\in \mathbb{P},g \in G}{\sum} \frac{[X_{ik}(g),X_{k}(p)]}{(z_i-g\cdot z_k)(z_k-p)} , \qquad \qquad 
  D=\underset{gh\neq 1}{\sum} \frac{[X_{ik}(g),X_{ki}(h)]}{(z_i- g \cdot z_k)(z_k-h\cdot z_i)}\]

\[E_j^1= \underset{g,h\in G}{\sum}\frac{[X_{ik}(g), X_{kj}(h)]}{(z_i-g\cdot z_k)(z_k-h\cdot z_j)}, \qquad \qquad 
E_j^2= \underset{g,h\in G}{\sum}\frac{[X_{ij}(g),X_{ki}(h)]}{(z_i-g\cdot z_j)(z_k-h\cdot z_i)},\]

\[E_j^3= \underset{g,h\in G}{\sum}\frac{[X_{ij}(g),X_{kj}(h)]}{(z_i-g\cdot z_j)(z_k-h\cdot z_j)}. \]

On va décomposer le termes $B,D$ et $E^2_j$. En appliquant (1)  du lemme \ref{rel forme}, on peut décomposer $E_j^2$ :
\begin{align*}
E_j^2=\underset{g,h}{\sum} \frac{[X_{ij}(g),X_{ki}(h)]}{(z_i-g\cdot z_j)(z_k-hg\cdot z_j)}+\frac{-[X_{ij}(g),X_{ki}(h)]}{(z_i-h^{-1}\cdot z_k)(z_k-hg\cdot z_j)}+\frac{[X_{ij}(g),X_{ki}(h)] }{(z_i-h^{-1}\cdot z_k)(z_k - h\cdot \infty )}.
\end{align*}
On note, dans l'ordre d'appariton, $E_j^{21},E_j^{22},E_j^{23}$ les termes de cette décomposition. En réindexant les termes de $E_j^3$, on trouve : 
\begin{align*}
E_j^3+E_j^{21}&=\underset{g,h\in G}{\sum} \frac{[X_{ij}(g),X_{ki}(h)+X_{kj}(hg)]}{(z_i-g\cdot z_j)(z_k-hg\cdot z_j)}.
\end{align*}
De même, en réindexant les termes de $E_j^{22}$ et en utilisant la relation (1) de la définition \ref{def AL}, on trouve :
\begin{align*}
E_j^{22}+E_j^1=\underset{g,h\in G}{\sum} \frac{[X_{ik}(g),X_{jk}(h^{-1})+X_{ji}(h^{-1}g^{-1})]}{(z_i-g\cdot z_k)(z_k-h\cdot z_j)}.
\end{align*}
Les relations (4) dans $\mathfrak{p}_n(G)(\mathbb{C})$ impliquent que $E_j^3+E_j^{21}=E_j^{22}+E_j^1=0$. Ainsi, on peut réduire l'équation (\ref{decomp omega wedge}) à : 
\begin{equation}\label{wedgeik1}
 \omega^i\wedge \omega^k+\omega^k\wedge \omega^i=(A+B+C+D+ \underset{j\vert j \notin\{i,k\}}{\sum}E_{j}^{23}) dz_i\wedge dz_k, \: \text{pour} \: i\neq k.
\end{equation}
Décomposons $B$. En appliquant (1) du lemme \ref{rel forme}  à $B$, pour $z=z_i,y=z_k$ et $z=p$, on obtient :
\begin{align*}
B= \underset{p,g}{\sum} \frac{[X_i(p),X_{ki}(g)]}{(z_i-p)(z_k-g\cdot p)}+\underset{p ,g}{\sum} \frac{-[X_i(p),X_{ki}(g)]}{(z_i-g^{-1}\cdot z_k)(z_k-g\cdot p)} +\underset{p,h}{\sum} \frac{[X_i(p),X_{ki}(g)]}{(z_i-g^{-1}z_k)(z_k-g\cdot \infty)}.
\end{align*}
On note comme avant $B_1,B_2,B_3$ les termes de cette décomposition. Pour chaque $q \in \mathcal{O}(p)$, on choisi un $g_q\in G$ tel que $g_q\cdot p=q$ et on note $G_{p}$ l'ensemble des $g_q$. Ainsi, on peut écrire : 
\[A+B_1= \underset{p}{\sum} \ \ \underset{g\in G_{p}}{\sum}\:\: \underset{h\in \mathrm{stab}(p)}{\sum}\:  \frac{[X_i(p),X_k(g\cdot p) +X_{ki}(gh)]}{(z_i-p)(z_k-g\cdot p)}.\]
Ce terme est nul à cause de la relation (7) de la définition \ref{def AL}. Ce qui réduit (\ref{wedgeik1}) à :
\begin{equation}\label{wedgeik2}
\omega^i\wedge \omega^k+\omega^k\wedge \omega^i=(B_2+B_3+C+D+ \underset{j\vert j \notin\{i,k\}}{\sum}E_{j}^{23}) dz_i\wedge dz_k, \: \text{pour} \: i\neq k.
\end{equation}
Il nous reste à faire la décomposition de $D$. Le lemme \ref{rel forme} donne :
\[D=\underset{gh\neq 1}{\sum} (w_1(g,h)-w_2(g,h)) +\underset{h,g}{\sum} \: \frac{[ X_{ik}(g),X_{ki}(h)]}{(z_i-h^{-1}\cdot z_k)(z_k - h\cdot \infty )},\]
où 
\[w_1(g,h)=  \frac{[X_{ik}(g),X_{ki}(h)]}{(z_i- g \cdot z_k)(z_k-hg\cdot z_k)}\quad \text{et} \quad w_2(g,h)=\frac{[X_{ik}(g),X_{ki}(h)]}{(z_i- h^{-1} \cdot z_k)(z_k-hg\cdot z_k)}.\]
On pose $D_1=\underset{gh\neq 1}{\sum} (w_1(g,h)-w_2(g,h)) $ et $D_2=D-D_1$. 
Remarquons que la relation en degré un dans $\mathfrak{p}_n(G)$ implique que la somme $B_3+D_2 +\underset{j\vert j \notin\{i,k\}}{\sum}E_{j}^{23}$ est nulle. Donc  : 
\begin{equation}\label{wedgeik3}
\omega^i\wedge \omega^k+\omega^k\wedge \omega^i=(B_2+C+D_1) dz_i\wedge dz_k, \: \text{pour} \: i\neq k.
\end{equation}
On va montrer que $B_2+C+D_1$ est nulle. On peut transformer $D_1$, puis utiliser la relation (1) dans $\mathfrak{p}_n(G)$  pour trouver :
\begin{align*}
D_1&=\underset{g}{\sum}\:\:\underset{h\neq 1}{\sum} (w_1(g,(gh)^{-1})-w_2(gh,g^{-1}))\\
 &= \underset{g}{\sum} \frac{1}{z_i-g\cdot z_k}\:\underset{h\neq 1}{\sum} [X_{ik}(g),X_{ik}(gh)]( \frac{1}{z_k-h\cdot z_k}+\frac{1}{z_k-h^{-1} \cdot z_k}).
\end{align*}
De plus, en utilisant la proposition \ref{lem orbite} et (2) du lemme \ref{rel forme}, $D_1$ se simplifie : 
\begin{align*}
D_1&=\underset{g}{\sum} \frac{1}{z_i-g\cdot z_k}\:\underset{q\in \mathfrak{p}}{\sum} \:\:\underset{h \in \mathrm{stab}(q) \setminus\{1\}}{\sum} [X_{ik}(g),X_{ik}(gh)]( \frac{1}{z_k-h\cdot z_k}+\frac{1}{z_k-h^{-1} \cdot z_k})\\
&= \underset{g}{\sum} \frac{1}{z_i-g\cdot z_k}\:\underset{q\in \mathfrak{p}}{\sum}( \frac{1}{z_k-q}+\frac{1}{z_k-\mathrm{at}(q)})\underset{h \in \mathrm{stab}(q) \setminus\{1\}}{\sum} [X_{ik}(g),X_{ik}(gh)]\\
&=\underset{g,p}{\sum} \: \underset{h \in \mathrm{stab}(p) \setminus\{1\}}{\sum} \frac{[X_{ik}(g),X_{ik}(gh)]}{(z_i-g\cdot z_k)(z_k-p)}.
\end{align*}
Par une réindexation de $B_2$ et l'utilisation de la relation (1) dans $\mathfrak{p}_n(G)$, on trouve :
\[B_2+C+D_1=\underset{g,p}{\sum} \underset{h\in \mathrm{stab}(p)}{\sum} \frac{[X_{ik}(g),X_k(p)+X_i(g\cdot p)+X_{ik}(gh)]}{(z_i-g\cdot z_k)(z_k-p)}.\]
Le numérateur de la fraction est nul en vertu de la relation (6) dans $\mathfrak{p}_n(G)$. Ce qui montre que $ B_2+C+D_1=0$. En injectant cette égalité dans (\ref{wedgeik3}), on trouve que pour tout $i\neq k$, $\omega^i \wedge \omega^k +\omega^k\wedge \omega^i=0$. Les éléments $\omega^i \wedge \omega^i$ étant nuls par définition de $\wedge$ et compte tenu de l'équation (\ref{decomp omega}), on obtient que $\omega \wedge \omega=0$, ce qui termine la démonstration de la proposition. 
\end{proof}
\subsection{Représentation de monodromie}\label{rep monodromie}
 On fixe $\underline{q}_n=(q_1,\cdots, q_n)$ un $n$-uplet appartenant à $\mathrm{Cf}_G(n,\PP^1_*)$. On a vu que la connexion ($\mathrm{Cf}_G(n,\PP^1_*)$, $\widehat{\mathrm{U}}\mathfrak{p}_n(G)(\mathbb{C})$, $\omega$) est plate. De plus, $\omega$ est à valeurs dans les primitifs de $\widehat{\mathrm{U}}\mathfrak{p}_n(G)(\mathbb{C})$. Donc une représentation de monodromie associée à cette connexion est à valeurs dans $\mathcal{G}(\widehat{\mathrm{U}}\mathfrak{p}_n(G)(\mathbb{C}))$, où $\mathcal{G}$ désigne le groupe des éléments diagonaux.\\\\ 
Soit $r : \tilde{C} \longrightarrow  \mathrm{Cf}_G(n,\PP^1_*)$ le revêtement universel de $\mathrm{Cf}_G(n,\PP^1_*)$, $\tilde{\underline{q}}_n$ un point de la fibre au-dessus de $\underline{q}_n$. On notera $ \rho_{\tilde{\underline{q}}_n} : \pi_1(\mathrm{Cf}_G(n,\PP^1_*), \underline{q}_n) \longrightarrow \mathcal{G}(\widehat{\mathrm{U}}\mathfrak{p}_n(G)(\mathbb{C}))$ la représentation de monodromie obtenue associée à $\tilde{\underline{q}}_n$.\\\\ Pour $i\neq j \in [1,n] $ et $g\in G$, on définit $\gamma_{ij}^g$ (resp. $\gamma_i^p$) comme un lacet dans $\PP^1$, basé en $q_i$ et qui sépare la sphère en deux disques ouverts, le premier contenant le point $g\cdot q_j$ et le deuxième contenant les points différents de $q_i$ et $ g\cdot q_j$ (resp. differents de $q_i$ et $ p$)  appartenant à l'ensemble $ (\underset{l\in [1,n]}{\cup} G \cdot q_l)\cup (\PP^1\setminus \PP^1_*)$. On oriente ces lacets comme en figure 10. On définit alors le lacet $x_{ij}^g$ (resp. $x_i^q$) de $\mathrm{Cf}_G(n,\PP^1_*)$ basé en $\underline{q}_n$ comme étant l'image de $\gamma_{ij}^g$ (resp. $\gamma_i^q$) sous l'application :
\begin{align*}
\PP^1_* \setminus \underset{k \vert k\neq i}{\cup}G\cdot q_k \longrightarrow & \mathrm{Cf}_G(n,\PP^1_*) \\
x \longrightarrow & (q_1, \cdots, q_{i-1},x, q_{i+1}, \cdots,q_n).
\end{align*}
On va étudier l'image des classes des $x_{ij}^g$ et $x_{i}^q$ par cette représentation.

\begin{center}
\begin{tikzpicture}[scale=0.4]
\draw[thick] (0,0) ellipse (4 and 4);

\draw [ thick,>=stealth,-<](-0.1,-3.1) .. controls (1,-2.25) and (1.9,-2) .. (2.65,-0.15);
\draw [thick,>=stealth,-<]  (2.65,-0.15) .. controls (3.1,1.5) and (1.65,1.5) .. (1.45,0); 
\draw [thick]  (1.45,0) .. controls (1.3,-1) and (0.6,-2) .. (-0.1,-3.1) node{$\bullet$} node[left,above]{$q_i$};

\draw (2.15,0) node{$\bullet$};
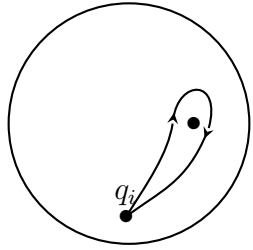
\node at (0,-5) {\textbf{Figure 1} : Le lacet $\gamma_{ij}^g$ (ou $\gamma_i^p$), dans $S^2$ (vue d'extérieur) ; };
\node at (0.1,-6.5) {le point non indexé est $g .q_j$ (ou $p$)};

\end{tikzpicture} \end{center}
\begin{proposition}\label{monodromie}
Pour $1\leq i<j\leq n$, $g\in G$, $k\in [1,n]$ et $p\in \PP^1\setminus \PP^1_*$, on a les égalités suivantes :
\[ \rho_{\tilde{\underline{q}}_n}(x_{ij}^g)= 1-2i\pi X_{ij}(g)+R_{ij}(g), \qquad \qquad \rho_{\tilde{\underline{q}}_n}(x_{k}^p)=1-2i\pi X_{k}(q)+R_k(p),  \]
où $R_{ij}(g)$ et $R_k(p)$ ne comportent que des termes de degré supérieur où égale à deux.
\end{proposition}
\begin{proof} 
On ne montrera l'égalité que pour $x_{ij}^g$, l'autre cas étant similaire. Pour calculer la monodromie, il faut s'intéresser à la solution $F$ de l'équation $dF=r^*(\omega)\wedge F$, avec $F(\tilde{\underline{q}}_n)=1$ (cf. section 1.3).  Soit $\tilde{x}_{ij}^g$ le chemin de $\tilde{C}$ relevant le lacet $ x_{ij}^g$ de point de départ $\tilde{\underline{q}}_n$. Son point d'arrivée est $x_{ij}^g\cdot \tilde{\underline{q}}_n$.\\
On a $\rho_{\tilde{\underline{q}}_n}(x_{ij}^g)=F(x_{ij}^g\cdot \tilde{\underline{q}}_n)$. Soit $F=F_0+F_1+ [deg\geq 2]$ le développement de $F$ en composantes homogènes. On sait que $F_0=1$ et on a $dF_1=r^*(\omega)$. On en déduit : 
\begin{align*}
\rho_{\tilde{\underline{q}}_n}(x_{ij}^g)&=1+F_1(x_{ij}^g\cdot \tilde{\underline{q}}_n)-F_1(\tilde{\underline{q}}_n)+[deg \geq 2]\\
&=1+\int_{\tilde{x}_{ij}^g} r^*(\omega) +[deg \geq 2]\\
&=1+ \int_{x_{ij}^g}\omega +[deg \geq 2]\\
&=1-2i\pi X_{ij}(g) +[deg \geq 2]
\end{align*}\end{proof}
\section{Rappels de topologie différentielle}
Cette partie est consacrée à des rappels de topologie différentielle. La notion de méridien est introduite d'une manière adaptée dans \ref{Gp,mer}. On termine \ref{Gp,mer} par une propriété de conjugaison entre méridiens. Enfin, la dernière sous-section \ref{stab,surf} de cette partie rappelle des propriétés sur les stabilisateurs de certaines actions de groupes finis sur les surfaces.\\\\Dans cette section les variétés sont sans bords sauf mention du contraire.

\subsection{Groupes fondamentaux et méridiens}\label{Gp,mer} Dans cette section, on définit (motivé par \cite{Kerv},\cite{SM}) une classe de lacets libres appelés méridiens et on compare les classes d'homotopie de deux méridiens.\\ 
Soit $M$ une variété différentielle connexe orientée et $N$ une sous-variété de $M$ connexe orientée de codimension 2.
On considère un disque $D \subset M$ transverse à $N$ tel que $D\cap N$ est réduit à un point $p$. Soit $(v_3,\cdots,v_{\mathrm{dim} M})$ une base du plan $\mathrm{T}_pN$ tangent à $N$ en $p$ compatible avec l'orientation de $N$. On oriente $D$ grâce à une base $(v_1,v_2)$ de $\mathrm{T}_pD$ choisie de manière que $(v_1,\cdots, v_{\mathrm{dim} M})$ soit une base de $\mathrm{T}_p M$ compatible avec l'orientation de $M$.
\begin{definition}\label{Def mer}
Soit $D$ comme dans le paragraphe précédent muni de l'orientation ci-dessus. Un méridien $m$ est un lacet $m :[0,1]\longrightarrow M$ qui est une paramétrisation du bord de $D$ respectant l'orientation. On dira que $m$ est un méridien au-dessus de $p$.
\end{definition}
\begin{proposition}\label{meridien}
Soit $m_1$ et $m_2$ deux méridiens. Il existe un chemin $\beta$ de $M\setminus N$ tel que $m_1$ est homotope (à extrémités fixées) à $\beta m_2 \beta^{-1}$ dans $M\setminus N$. En particulier, si $m_1$ et $m_2$ ont le même point de base $*\in M\setminus N$, alors $m_1$ et $m_2$ sont conjugués dans $\pi_1(M\setminus N,*)$.  
\end{proposition}
\begin{proof}
D'abord, on vérifie que le lemme est vrai pour $M=\R^{n}\times D$ et $N=\R^n\times \{0\}$ avec $D$ un disque centré en zéro : il suffit de montrer que tout méridien est conjugué par un chemin à un mériden à valeurs dans $\{0\}\times D$.  Soit $m$ un méridien paramétrisant le bord de $D_1$ comme dans la définition \ref{Def mer}. On note $q$ l'unique point d'intersection (transverse) de $D_1$ avec $N$. Notons $\pi^{(2)}$ la projection naturelle de $\R^{n}\times D$ sur $\{0\}\times D$. L'intersection de $D_1$ avec $N$ étant transverse, on sait qu'il existe un voisinage $V_q$ de $q$ dans $D_1$  tel que $\pi^{(2)}_{\vert V_q}$ est injective. En utilisant une isotopie $I: D_1\times [0,1] \to D_1$ qui fixe $q$ et qui envoie $D_1$ dans $V_q$, puis l'isotopie : 
$$ (1-t) x+t\pi^{(2)}(x) ,\text{   pour $x\in V_q$ et $t\in [0,1]$ }, $$
on trouve que $m$ est conjugué par un chemin à un méridien dans $\{0\}\times D$. Ceci montre le lemme dans le cas $M=\R^{n}\times D$ et $N=\R^n\times \{0\}$.\\
 Passons au cas général. Soit $\pi :V \longrightarrow N$ un fibré en disques qui est un voisinage de $N$ dans $M$ et $p$ un point de $N$. On peut trouver un voisinage $U_p$ de $p$ dans $N$, difféomorphe à $\R^{\mathrm{dim} N}$ et au-dessus duquel $V$ est trivial. On a $\pi^{-1}(U_p) \simeq \R^{\mathrm{dim} N} \times D$ et $U_p$ s'identifie à $\R^{\mathrm{dim} N}\times\{0\}$. \\
Tout méridien $m$ au-dessus d'un point de $U_p$ est conjugué par un chemin dans $M\setminus N$ à un méridien de  $\pi^{-1}(U_p)$. En effet, il suffit de rétrécir le disque contenant $m$ par une isotopie. Par conséquent, compte tenu du cas $M=\R^{n}\times D$ et $N=\R^n\times \{0\}$, tous les méridiens au-dessus des points de $U_p$ sont conjugués les uns aux autres par des chemins dans $M\setminus N$.\\
On a donc une relation d'équivalence $\mathcal{R}$ sur $N$ donnée par : 
\[x\mathcal{R}y \iff \text{les méridiens au-dessus de $x$ sont conjugués à ceux au-dessus de $y$},\]
dont les classes d'équivalence sont ouvertes. Comme $N$ est connexe, on n'a qu'une seule classe. Ceci montre la proposition.
\end{proof}

\begin{corollary}\label{conjMeridian}
Si $\gamma$ est un chemin du point de base $*_1$ de $m_1$ vers celui de $m_2$ alors $m_1$ et $\gamma m_2 \gamma^{-1}$ sont conjugués dans $\pi_1(M\setminus N,*_1)$.
\end{corollary}
\subsection{Groupe fini agissant sur une surface}\label{stab,surf} Soit $S$ une surface orientée, compacte, sans bord, munie d'une action fidèle par difféomorphismes (conservant l'orientation) d'un groupe fini $H$. On note $S_*$ l'ensemble des points de $S$ à stabilisateur trivial pour l'action de $H$ sur $S$. On suppose que $S\setminus S_*$ est fini (ceci est automatique car $H$ préserve l'orientation).
\begin{proposition}\label{disc,stab}
Soit $p$ un point de $S$. Le stabilisateur $\mathrm{stab}(p)$ de $p$ est un sous-groupe cyclique de $H$. De plus, il existe un disque $D_p$ autour de $p$ tel que : 
\begin{enumerate}
\item Le groupe $\mathrm{stab}(p)$ agit sur $D_p$, cette action est équivalente à l'action par multiplication de $\mu_N$ sur $\C$, où $N=\vert \mathrm{stab}(p)\vert$. 
\item $D_p\cap (H\cdot q)=\mathrm{stab}(p) \cdot q$, pour $q\in D_p$,
\end{enumerate} 
\end{proposition}
\begin{proof}
Soit $\mathbf{m}$ une métrique riemannienne $H$-invariante sur $S$ (on peut en obtenir en faisant la moyenne d'une métrique quelconque). Le groupe $\mathrm{stab}(p)$ agit naturellement par isométries sur $T_pS$; on a donc un morphisme $T_p : \mathrm{stab}(p)\longrightarrow \mathrm{O}(T_pS)$, injectif d'après la proposition 3.11 de \cite{FP}. Plus précisément, $\mathrm{stab}(p)$ est inclus dans $\mathrm{SO}(T_pS)\simeq S^1$ car $H$ préserve l'orientation. Ainsi, $\mathrm{stab}(p)$ est isomorphe à $\mu_N$ avec $N=\vert \mathrm{stab}(p) \vert$. De plus, pour tout disque $D_r\subset T_p S$ centré en zéro de rayon $r$ inférieur au rayon d'injectivité de l'application exponentielle $\mathrm{Exp}_p :T_pS \longrightarrow S$, l'exponentielle établit un difféomorphisme entre $D_r$ et $D_p(r)=\mathrm{Exp}_p(D_r)$. De plus, $\mathrm{Exp}_p$ vérifie :
\[
\forall h\in \mathrm{stab}(p),\qquad \mathrm{Exp}_p\circ T_p(h)=h\circ \mathrm{Exp}_p.
\]
En effet, si $\gamma$ est la géodésique d'origine $p$ et de vecteur $v\in T_pS$, alors $h\circ \gamma$ est la géodésique de vecteur $T_p(h)(v)$ d'origine $h(p)=p$. On voit donc que $\mathrm{stab}(p)$ agit sur $D_p(r)$ et que cette action est équivalente à celle de $\mu_N$ par multiplication sur $D_r$. Ce qui montre (1).\\\\ Montrons que certains des disques $D_p(r)$ satisfont également l'assertion (2). Si le contraire était vrai, on aurait pour tout $k\geq 1$, existence de $q_k\in D_p(\frac{1}{k})$ et $g_k\in H\setminus\mathrm{stab}(p)$ tels que $g_k(q_k)\in D_p(\frac{1}{k})$. Comme $H$ est fini, on peut extraire une suite $(q_{k_i})_{i\geq 0}$ de $(q_k)_{k\geq 0}$ telle que $(g_{k_i})_{i\geq 0}$ est constante égale à un certain $g\in H\setminus\mathrm{stab}(p)$.  Comme $(q_{k_i})_{i\geq 0}$ tend vers $p$, en passant à la limite, on trouve la contradiction $p=g(p)$.
\end{proof}
En particulier, pour $p$ un point à stabilisateur trivial, $D_p$ ne contient pas deux éléments d'une même orbite.
\section{Etude de groupes fondamentaux}
La présente section a pour but de donner des générateurs de $\pi_1(\mathrm{Cf}_H(n,S_*))$ et des relations entres ces générateurs, pour $S$ et $H$ comme dans la section 2. Pour ce faire, on définit dans la sous-section \ref{Genpi1} des lacets qui engendrent $\pi_1(\mathrm{Cf}_H(n,S_* ))$. Une identification de ces générateurs à des méridiens (sous-section \ref{MerMer}), nous permet de décrire l'action de $H^n$ sur les classes de conjuguaison de $\pi_1(\mathrm{Cf}_H(n,S_*))$ (sous-section \ref{Actionconj}). Enfin, dans la sous-section \ref{Ref}, on établit des relations quadratiques dans $\pi_1(\mathrm{Cf}_H(n,S_*))$.

\subsection{Générateurs de $\pi_1(\mathrm{Cf}_H(n,S_*),\underline{q}_n) $}\label{Genpi1}
Dans cette partie, on construit une famille de générateurs de $\pi_1(\mathrm{Cf}_H(n,S_*)) $. En particulier, on trouve une famille génératrice de l'espace de configurations d'orbites $\mathrm{Cf}_G(n,\PP^1_*)$ introduit dans la section 1.\\\\
Fixons un point $\underline{q}_n=(q_1,\cdots, q_n) \in \mathrm{Cf}_H(n,S_*)$. Pour définir des lacets, on voit $S$ comme une surface dans $\R^3$  de la manière usuelle. On oriente $S$ en prenant la normale vers l'extérieur.\\\\
Choisissons $i\in[1,n]$ et posons $E_i=(S\setminus S_*)\cup (\cup_{j\neq i} H \cdot q_j)$. Pour $p \in E_i$, on considère un disque fermé $\bar{D}$ contenant $q_i$ dans son bord et $p$ dans son intérieur tel que $\bar{D}\cap E_i=\{q_i,p\}$. On note $\gamma_i(p)$ le lacet basé en $q_i$ paramétrisant le bord de $\bar{D}$ dans le sens inverse à celui induit par la normale.  Si $S$ est de genre $g(S)>0$, on considère les lacets $\gamma_i^1 \cdots \gamma_i^{2g(S)}$, générateurs usuels de $\pi_1(S,q_i)$ évitant les points de $E_i$. 
\begin{definition} Pour $ i\in[ 1,n]$, $j\in [i+1,n]$, $m\in[1,2g(S)]$, $h\in H$ et $q \in S\setminus S_*$, les lacets $x_{ij}(h)$, $x_i(q)$ et $x_i^{m}$ de $\mathrm{Cf}_H(n,S_*)$ sont définis par :
\begin{enumerate}
\baselineskip 20pt
\item[--] $x_{ij}(h)(t) =(q_{n-k+1},\cdots ,q_{i-1},\gamma_i(h\cdot q_j)(t), q_{i+1},\cdots, q_n)$, pour tout $t\in [0,1]$;
\item[--] $x_i(q)(t)=(q_{n-k+1},\cdots ,q_{i-1},\gamma_{i}(q)(t), q_{i+1},\cdots, q_n)$, pour tout $t\in [0,1]$;
\item[--] $x_i^{m}(t)=(q_{n-k+1},\cdots ,q_{i-1},\gamma_{i}^m(t), q_{i+1},\cdots, q_n)$, pour tout $t\in [0,1]$;
\end{enumerate}
où les $\gamma_i(p)$ et les $\gamma_i^m$ ont été définis dans le paragraphe précédent. 
\end{definition}
Remarquons que ceci définit des lacets à conjugaison près dans $\pi_1(\mathrm{Cf}_H(n,S_*))$. On fixe un choix de tels lacets pour la suite.
Pour montrer que les lacets ainsi définis engendrent $\pi_1(\mathrm{Cf}_H(n,S_*)) $, on va utiliser le résultat suivant :
\begin{theorem}[\cite{CKX}]
Si $M$ est une variété sans bord munie d'une action libre d'un groupe fini $H$, alors, pour $n\geq 2$, la projection $p^{n}: \mathrm{Cf}_H(n,M) \longrightarrow \mathrm{Cf}_{H}(n-1,M)$ sur les $n-1$ dernières composantes est une fibration localement triviale.
\end{theorem} 
On peut trouver une preuve de ce théorème dans \cite{DC} (Théorème 2.1.2). Vu les hypothèses imposées sur les couples $(S,H)$, ce théorème s'applique aux $\mathrm{Cf}_H(n,S_*)$.
\begin{proposition}\label{Gen Punc}
Le groupe $\pi_1(\mathrm{Cf}_H(n,S_*))$ est engendré par les classes des $x_{ij}(h)$ $(1 \leq i<j \leq n\: et \: h\in H)$, des $x_i(q)$ $(i\in [1,n]\: et \: q\in S\setminus S_*)$ et des $x_i^{m}$ $(i\in[1,n]\: et \: m\in[1,2g(S)])$ .
\end{proposition}
\begin{proof}
On va montrer la proposition par récurence. La proposition est vraie pour $k=1$. Supposons qu'elle soit vraie au rang $n-1$. Considérons la fibration : 
\[(\mathrm{Fib}_n,\underline{q}_{n}) \longrightarrow (\mathrm{Cf}_H(n,S_*), \underline{q}_{n}) \overset{p^ n}{\longrightarrow}(\mathrm{Cf}_H(n-1,S_*), \underline{q}_{n-1}),\]
où $\underline{q}_{n-1}=(q_2,\cdots,q_n)$ et $\mathrm{Fib}_{n}$ est la fibre correspondante. Observons la suite exacte de cette fibration :
\[\cdots \longrightarrow \pi_1(\mathrm{Fib}_{n},\underline{q}_{n}) \longrightarrow \pi_1(\mathrm{Cf}_H(n,S_*), \underline{q}_{n}) \longrightarrow \pi_1(\mathrm{Cf}_H(n-1,S_*), \underline{q}_{n-1}) \longrightarrow 1 \longrightarrow \cdots. \]
Soit $x_{ij}(h),x_i(q)$ et $x_i ^m$ les lacets de $\mathrm{Cf}_H(n,S_*)$ définis plus haut. Les images de ces lacets par $p^n$ sont des lacets analogues de $\mathrm{Cf}_H(n-1,S)$, dont on sait par hypothèse de récurrence qu'ils engendrent le groupe. Il s'ensuit que les classes des  :
\[p^{n}(x_{ij}(h)),  p^{n}(x_i(q)) \: \text{et celles des} \:  p^n(x_i^{m}), \] 
pour $ i\in[2,n], j\in[i+1,n] ,h\in H, q\in S\setminus S_*$ et $m\in[1,2g(S)]$, engendrent le groupe $\pi_1(\mathrm{Cf}_H(n-1,S_*), \underline{q}_{n-1})$.\\
Enfin, on sait que les $x_{1,j}(h),x_{1}(q) $ et les $x_{1}^{m}$, pour $j\in [2,n], h\in H,q\in S\setminus S_*$ et $m\in [1,2g(S)]$, engendrent le groupe fondamental de $\mathrm{Fib}_n\simeq \{p\in S_* \vert p \notin H\cdot q_j ,\: j=2, \cdots ,n\}$. La proposition est alors une conséquence du résultat suivant : Si $K\longrightarrow G \longrightarrow H \longrightarrow 1$ est une suite exacte de groupes, $F_{K}$ est une famille génératrice de $K$ et $F_G$ est une famille de $G$ telle que $\mathrm{Im}(F_G\subset G \longrightarrow H)$ est génératrice de $H$, alors $F'_K\cup F_G$ ($F'_K$ est l'image de $F_K$ dans $G$) est une famille génératrice de $G$.
\end{proof}
Pour $S=\PP^1$ et $H=G$, on obtient :
\begin{proposition}\label{Gen}
Le groupe $\pi_1(\mathrm{Cf}_G(n,\PP^1_*))$ est engendré par les classes des $x_{ij}(g)$ $(1 \leq i<j \leq n\: et \: g\in G)$ et des $x_i(q)$ $(i\in [1,n]\: et \: q\in \PP^1\setminus \PP^1_*)$.

\end{proposition}

\begin{remarque}
Si le point $\underline{q}_n$ utilisé ici est le même que celui de la sous-section \ref{rep monodromie}, alors les lacets $x_{ij}(g)$ et $x_i(q)$ sont respectivement conjugués dans $\pi_1(\mathrm{Cf}_G(n,\PP^1_*))$ à $x_{ij}^g$ et $x_i^q$ de \ref{rep monodromie}.
\end{remarque}

\subsection{Les tresses $x_{ij}(g)$ en tant que méridiens}\label{MerMer} Dans ce paragraphe, on identifie les lacets introduits dans la sous-section \ref{Genpi1} à des méridiens pour des variétés liées à $\mathrm{Cf}_H(n,S_*)$.\\\\
Pour $h\in H$, $k\in [1,n]$, $i,j \in [1,n]$ avec $i<j$, et $q \in S\setminus S_*$, on pose $ D(i,j,h):=\{(y_1,\cdots,y_n) \in S^n \vert y_i= hy_j\}$ et $D(k,q):=\{(y_1,\cdots,y_n) \in S^n \vert y_k=q\}$. Pour $\alpha \in E:=\{(i,j,h) \vert 1\leq i<j\leq n , h\in H\}\cup ([1,n]\times (S\setminus S_*))$, on définit les deux sous-variétés de $S^n$ : 
\[ M(\alpha)=S^n \setminus (  \underset{\beta \in E\setminus \{\alpha\}}{\cup} D(\beta)) \qquad\text{et} \:\qquad  N(\alpha)=D(\alpha) \cap M(\alpha).\]

$N(\alpha)$ est une sous-variété de codimension réelle deux dans $M(\alpha)$ et $\mathrm{Cf}_H(n,S_*)$ est le complémentaire de $N(\alpha)$ dans $M(\alpha)$. De plus, $M(\alpha)$ et $N(\alpha)$ sont connexes.\\ 
Remarquons que $x_{ij}(h)$ est le bord d'un disque de $M(i,j,h)$ coupant $N(i,j,h)$ en un unique point $\underline{q}'_n$ : 
\[\underline{q}_n'=(q_1,\cdots,q_{i-1},h\cdot q_j, q_{i+1},\cdots,q_n).\]
Cette intersection est transverse. On oriente $M(i,j,h)$  et $N(i,j,h)$ d'une manière que $x_{ij}(h)$ devient un méridien au sens de la définition \ref{Def mer}. On fait la même chose pour que $x_k(q)$ soit un méridien au sens de \ref{Def mer} (pour $M=M(k,q)$ et $N=N(k,q)$). 
\begin{definition}
Soit $\alpha$ comme ci-dessus. Un lacet $m$ de $\mathrm{Cf}_H(n,S_*)$ est dit méridien pour $\alpha$ si $m$ satisfait la définition \ref{Def mer}, pour $M=M(\alpha)$ et $N=N(\alpha)$ munies de l'orientation ci-dessus.
\end{definition}

\begin{proposition}\label{conj mer}
Tout méridien pour $(i,j,h)$ est conjugué par un chemin au lacet $x_{ij}(h)$ et tout méridien pour $(k,q)$ est conjugé par un chemin à $x_k(q)$.
\end{proposition}

\begin{proof}
Cette proposition est une conséquence de la proposition \ref{meridien}.
\end{proof}

\subsection{Action de $H^n$ sur les classes de conjugaison de $\pi_1(\mathrm{Cf}_H(n,S_*))$}\label{Actionconj} Dans cette sous-section, on décrit l'action de $H^n$ sur les classes de conjugaison des $x_{ij}(h)$ et des $x_k(q)$ de $\pi_1(\mathrm{Cf}_H(n,S_*))$ donnés dans la proposition \ref{Gen}.\\\\
On a vu que $H^n$ agit naturellement sur $\mathrm{Cf}_H(n,S_*)$ et on avait choisi un point de base $\underline{q}_n$ de $\mathrm{Cf}_H(n,S_*)$. On considère l'ensemble des couples $(\underline{h},C)$ où $\underline{h}\in H^n$ et $C$ est une classe d'homotopie d'un chemin dans $\mathrm{Cf}_H(n,S_*)$ de point de départ $\underline{q}_n$ et de point d'arrivée $\underline{h}\cdot \underline{q}_n$. On munit cet ensemble de la multiplication : $(\underline{h},C) \cdot (\underline{h}',C')= (\underline{h}\underline{h}',  C(\underline{h}\circ C'))$, où $ C(\underline{h}\circ C')$ est le résultat de la concaténation des lacets $C$ et $(\underline{h}\circ C')$ $(\underline{h}\circ C'$ désignant la composée du lacet $C'$ à valeurs dans $\mathrm{Cf}_H(n,S_*)$ avec le difféomorphisme de $\mathrm{Cf}_H(n,S_*)$ induit par $\underline{h})$. On obtient ainsi le groupe $\pi_1^{\mathrm{orb}}(\mathrm{Cf}_H(n,S_*)/H^n,\bar{\underline{q}}_n) $ (qui est exactement le groupe fondamental de $\mathrm{Cf}_H(n,S_*)/H^n$) qui s'insère dans la suite exacte :
\begin{equation}\label{suite Orb}
1 \longrightarrow \pi_1(\mathrm{Cf}_H(n,S_*), \underline{q}_n) \overset{\varphi}{\longrightarrow}  \pi_1^{\mathrm{orb}}(\mathrm{Cf}_H(n,S_*)/H^n,\bar{\underline{q}}_n)  \overset{\theta}{\longrightarrow} H^n \longrightarrow 1. 
\end{equation}
où $\varphi(\gamma)=(1,\gamma)$ et $\theta(\underline{h},C)=\underline{h}$. Toute section ensembliste $\sigma$ de cette suite donne une application ensembliste :
\begin{equation}\label{section}
f_\sigma :H^n {\longrightarrow} \mathrm{Aut}(\pi_1(\mathrm{Cf}_H(n,S_*), \underline{q}_n)),
\end{equation}
qui à $\underline{h}$ associe la conjugaison par $\sigma(\underline{h})$. \\\\
Pour $i\in [1,n]$, $h_i$ désignera l'image de $h\in H$ par l'inclusion canonique de $H$ dans $H^n$ qui envoie $H$ en $i$-ème position. La relation de conjugaison dans $\pi_1(\mathrm{Cf}_H(n,S_*),\underline{q}_n)$ sera notée $\sim$.

\begin{proposition}\label{G.Lacet}On a :
\begin{align}  f_\sigma(h_r)( x_{ij}(g)) \sim \left\{
    \begin{array}{ll}
\:x_{ij}(hg) & si \:r=i \\ \: x_{ij}(gh^{-1})  & si \:r=j \\ \: x_{ij}(g) & sinon
    \end{array},\right.
 \quad f_\sigma(h_r)( x_{k}(q)) \sim \left\{
    \begin{array}{ll}
\:x_{k}(h\cdot q) & si \:r=k \\ 
\: x_{k}(q)  &  sinon
    \end{array} \right. 
\end{align}

pour $1 \leq i<j \leq n $, $k,r\in [1,n]$, $g,h\in H$ et $q\in S\setminus S_*$. 
\end{proposition}
\begin{proof}Soit $\sigma, r,i,j,g,h$ comme dans l'énoncé. On sait que $\sigma(h_r)=(h_r,C_r)$ avec $C_r$ une classe d'un chemin reliant $\underline{q}_n$ à $h_r\cdot \underline{q}_n$. Par définition, on a :
\begin{equation}\label{Calcul orb}
f_{\sigma}(h_r)(x_{ij}(g))=\sigma(h_r)(1,x_{ij}(g))\sigma(h_r)^{-1}=(1, C_r( h_r \circ x_{ij}(g)) C_r^{-1}). \end{equation}
Le membre de droite s'identifie à $ C_r( h_r \circ x_{ij}(g)) C_r^{-1}$ dans $\pi_1(\mathrm{Cf}_H(n,S_*),\underline{q}_n)$. Posons $h'(r)=hg$ si $r=i$, $h'(r)=gh^{-1}$ si $r=j$ et $h'(r)= g$ sinon.\\
Comme $H$ est un difféomorphisme de $S$ conservant l'orientation et comme $h_r(N(i,j,g)=N(i,j,h'(r))$, on trouve que $h_r \circ x_{ij}(g)$ est un méridien pour $(i,j ,h'(r))$. Donc, d'après le corollaire \ref{conjMeridian} $f_{\sigma}(h_r)(x_{ij}(g))$ est un conjugué de $x_{ij}(h'(r))$. Ce qui montre l'égalité de gauche. De façon similaire, on démontre l'égalité de droite.  
\end{proof}

Enfin, l'application $f_{\sigma}$ induit une action de $H^n$ sur l'ensemble des classes de conjuguaison de $\pi_1(\mathrm{Cf}_H(n,S_*),\underline{q}_n)$, autrement dit : 
\begin{equation}
f_\sigma(\underline{h})(a) \sim f_\sigma(\underline{h})(a'), \:\text{si} \: a\sim a'.
\end{equation}
Cette action est indépendante de $\sigma$. 

\subsection{Relations dans  $\pi_1(\mathrm{Cf}_H(n,S_*),\underline{q}_n)$}\label{Ref}
Dans cette partie on donne des relations entre les $x_{ij}(g),x_k(q)$ : une première famille générale, et une deuxième famille valable dans le cas où $S=S^2$.\\\\
Considérons l'espace $\mathrm{Cf}_{\mu_N}(2,\C^\times)$ où le groupe $\mu_N$ est le groupe des racines $N$-ème de l'unité agissant par multiplication sur $\C$ :

\begin{lemma}\label{DT} Pour $k\in [1,N-1]$ et $j\in [1,2]$, il existe un $x_{12}'(k)$ conjugué à $ x_{12}(e^{\frac{2ik\pi}{N}})$ et $x'_j(0)$ conjugué à $x_j(0)$ dans $\pi_1(\mathrm{Cf}_{\mu_N}(2,\C^\times),(2,1))$ tel que :
\baselineskip 18pt 
\begin{enumerate}
\item les deux lacets $\beta:=x_{12}'(1) x_{12}'(2)\cdots x_{12}'(N-1)x'_1(0)$ et $x'_{2}(0)$ commutent.
\item les deux lacets $\alpha:=x'_2(0)x_{12}'(1) x_{12}'(2)\cdots x_{12}'(N-1)x'_1(0)$ et $x'_{12}(1)$ commutent.

\end{enumerate}

\end{lemma}

\begin{proof}On choisit des $x'_{12}(k)$ et les $x'_i(0)$ tels que $\beta$ soit égale à $(\gamma_1,1)$ et $x'_2(0)$ soit égale à $(2,\gamma_2) $, où $\gamma_1, \gamma_2$ sont représentés dans la figure 3 : 
\begin{center}

\begin{tikzpicture}[scale=0.7]

\draw[thick]  (0,0) circle (1);
\draw [thick]  (0,0) circle (2);

\node at (0,1) {$>$};
\node at (0,2) {$>$};
\node at (0,0) {$\bullet$};
\node at (0,0) [right]{$0$};
\node at (1,0) {$\bullet$};
\node at (1,0)  [right]{$1$};
\node at (2,0) {$\bullet$};
\node at (2,0) [right] {$2$};
\node at (0,-2.5) {$\textbf{Figure 2 :}$ Les lacets $\gamma_1$ et $\gamma_2$  };
\end{tikzpicture}
\end{center}
Les lacets $\beta$ et $x'_2(0)$ ainsi choisis commutent dans $\pi_1(\mathrm{Cf}_{\mu_N}(2,\C^\times),(2,1))$. En effet, il sont à image dans $\mathrm{Im}(\gamma_1)\times \mathrm{Im}(\gamma_2) \subset \mathrm{Cf}_{\mu_N}(2,\C^\times) $. Ce qui montre (1). Montrons (2). 
 Soit $H_0 : [0,1]\times [0,2]\longrightarrow \mathrm{Cf}_{\mu_N}(2,\C^\times)$ une homotopie entre $x_{12}(1)*$ et $*x_{12}(1)$, où $*$ est le lacet constant (voir figure 2) . 
\begin{center}
\begin{tikzpicture}[scale=0.7]
\draw[ultra thick] (0.2,0.9) .. controls (-1.75,0.4) and (-1.75,0.05) .. (-1.3,-0.15);
\draw[ultra thick]  plot[smooth, tension=.7] coordinates {(-1.05,0.9) (-1.05,0.6) };
\draw[ultra thick]   plot[smooth, tension=.7] coordinates {(-1.05,0.3) (-1.05,-0.55)};
\draw[ultra thick] (-0.8,-0.3) .. controls (-0.7,-0.3) and (0.2,-0.55) .. (0.2,-0.55);

\node at (0.2,0.9) {$\bullet$};
\node [above] at (0.2,0.9) {$2$};
\node at (-1.05,0.9) {$\bullet$};
\node [above] at (-1.05,0.9) {$1$};
\node at (-1.05,-0.55) {$\bullet$};
\node at (0.2,-0.55) {$\bullet$};

\draw[ultra thick] (3.15,-0.55) .. controls (1.2,-1.05) and (1.2,-1.4) .. (1.65,-1.6);
\draw[ultra thick]  plot[smooth, tension=.7] coordinates {(1.9,-0.55) (1.9,-0.85) };
\draw[ultra thick]   plot[smooth, tension=.7] coordinates {(1.9,-1.15) (1.9,-2)};
\draw[ultra thick] (2.15,-1.75) .. controls (2.25,-1.75) and (3.15,-2) .. (3.15,-2);

\draw[ultra thick]  plot[smooth, tension=.7] coordinates {(-1.05,-0.55) (-1.05,-2) };
\node at  (-1.05,-2){$\bullet$};
\node at (-1.05,-2)  [below] {$1$};
\draw[ultra thick]  plot[smooth, tension=.7] coordinates {(0.2,-0.55) (0.2,-2) };
\node at (0.2,-2) {$\bullet$};
\node  at (0.2,-2) [below] {$2$};

\node at (1.9,-2) {$\bullet$};
\node [below] at (1.9,-2) {$1$};
\node  at (3.15,-2) {$\bullet$};
\node [below] at (3.15,-2) {$2$};

\draw[ultra thick]  plot[smooth, tension=.7] coordinates {(1.9,0.9) (1.9,-0.55) };

\node at (1.9,-0.55) {$\bullet$};
\node at (1.9,0.9) {$\bullet$};

\node [above] at (1.9,0.9) {$1$};

\draw[ultra thick]  plot[smooth, tension=.7] coordinates { (3.15,-0.55)(3.15,0.9)  };

\node   at (3.15,0.9) {$\bullet$};
\node  [above] at (3.15,0.9) {$2$};
\node at (3.15,-0.55) {$\bullet$};

\node at (4.6,0.9) {$t=0$};
\node at (4.6,-2) {$t=2$};
\node at (4.6,-0.55) {$t=1$};

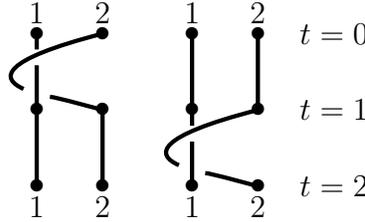
\node at (1.5,-3) {\textbf{Figure 3 :} Graphes en fonction du temps de $H_0(0,\cdot)$ (gauche)};
\node at (1.5,-3.7) {et de $H_0(1,\cdot)$ (droite)};
\end{tikzpicture}
\end{center}
L'homotopie tordue $H: [0,1]\times [0,2] \longrightarrow \mathrm{Cf}_{\mu_N}(2,\C^\times)$ définie par : 
\[H(s,t)=e^{-2i\pi(s+t) \mathcal{I}_{[1,2]}(s+t)} H_0(s,t) ,\]  où $\mathcal{I}_{[1,2]}$ est l'indicatrice de l'intervalle $[1,2]$, est une homotopie entre  $H(0,\cdot)=x_{12}(1)\alpha'$ et $H(1,\cdot)=\alpha'x_{12}(1) $, où $\alpha'$ est défini par  $\alpha'(t)=(2e^{-2i\pi t},e^{-2i\pi t})$ pour $t\in[0,1]$. Or, $\alpha'$ est homotope à un lacet du même type que $\alpha$. Ce qui montre (2). 
\end{proof}
\begin{remarque}
Les relations du lemme \ref{DT} ressemblent aux relations : 
 $$(X_{02}x_{12}(0)\cdots x_{12}(N-1)X_{01},x_{12}(0) )=(X_{02}x_{12}(0)\cdots x_{12}(N-1)X_{01},X_{02} )=1$$
de la proposition 1.1 de \cite{BE}.
Ce type de relations peut aussi apparaître en genre supérieur dans beaucoup de cas où l'on a une courbe simple stable de $S$ pour l'action de $H$. Par exemple, quand on fait agir le groupe $\mu_N$ sur le tore $S^1\times S^1$, on retrouve les relations (\ref{stab1}) et (\ref{stab2}) dans lesquelles $x_i(g\cdot q)$ et $x_j(p)$ sont remplacés par des lacets $(\text{des $x_i^m$ et $x_j^m$ bien choisis})$ représentés par des courbes simples stables. \end{remarque}
On emploie dans la proposition qui suit l'expression ''$(a_1\vert \cdots \vert a_r\: , \: b_1 \vert \cdots \vert b_s)^\sim=1$'' pour dire que pour $a_i,b_j \in \pi_1(\mathrm{Cf}_H(n,S_*),\underline{q}_n)$ $(i \in [1,r], j\in [1,s])$, ils existent $a_i'\sim a_i , b'_j \sim b_j$ $(i,j$ dans les mêmes ensembles) tels que : \[(a_1' \cdots a_r' , b_1' \cdots b_s' )=1,\]
où $(a,b)=aba^{-1}b^{-1}$.
\begin{proposition}\label{Rel Quad}
On a les relations suivantes dans le groupe  $\pi_1(\mathrm{Cf}_H(n,S_*),Q_n) $ :

\begin{equation}\label{2indices}
 (x_{ij}(g),x_{kl}(h))^\sim=(x_{il}(g), x_{jk}(h))^\sim=(x_{ik}(g), x_{jl}(h))^\sim=1 ,
\end{equation}
pour $1\leq i<j<k<l \leq n$ et $g,h \in H$,

\begin{align}\label{3indices}
(x_{ij}(g),x_{ik}(gh) \vert x_{jk}(h))^\sim&=( x_{jk}(h),x_{ij}(g)\vert x_{ik}(gh))^\sim\endline
&=(x_{ik}(gh), x_{jk}(h)\vert x_{ij}(g))^\sim=1,
\end{align}
\begin{align}\label{3indicesp}
(x_{i}(p), x_{jk}(g) )^\sim=(x_{j}(p), x_{ik}(g) )^\sim=(x_{k}(p), x_{ij}(g) )^\sim=1,
\end{align}
pour $1\leq i<j<k \leq n$, $g,h \in H$ et $p\in S\setminus S_*$,
\begin{equation}\label{2ind-orb}
(x_{i}(p), x_{k}(q) )^\sim=1,
\end{equation}
pour $1\leq i<k \leq n$, $p$ et $q$ dans  $S\setminus S_*$ avec $q\notin \mathcal{O}(p)$.
\begin{equation}\label{stab1}
 (x_j(p)\vert x_{ij}(gh^0)\vert x_{ij}(gh^1)\vert \cdots\vert x_{ij}(gh^{\vert \mathrm{stab}(p)\vert-1} )\vert x_i(g\cdot p),x_{ij}(g))^\sim=1,\\
\end{equation}
\begin{equation}\label{stab2}
 (x_{ij}(gh^0)\vert x_{ij}(gh^1)\vert \cdots\vert x_{ij}(gh^{\vert \mathrm{stab}(p)\vert-1} )\vert x_i(g\cdot p),x_j(p))^\sim=1,\\
\end{equation}
\\
pour $1\leq i<j\leq n$, $g\in H$, $p\in S\setminus S_*$ et $h$ un générateur de $\mathrm{stab}(p)$,

\begin{equation}\label{deg1}
 (X_{u^i_1}\cdots X_{u^i_r})=1,
\end{equation}\\
pour tout $i\in [1,n]$, $S=S^2$, $u^i_1,\cdots,u^i_r$ est une énumeration des éléments de $\{(k,l,g) \in [1,n]^2\times G \vert k<l , i\in\{k,l\} \} \cup(\PP^1\setminus \PP^1_*)$, $X_{u^i_\alpha}$ est un conjugué de $x_{kl}(g)$ si $u^i_\alpha=(k,l,g)$ ou un conjugué de $x_{i}(q)$ si $u^i_\alpha=q \in (\PP^1\setminus \PP^1_*)$.
\end{proposition}
\begin{proof}
On a vu que les $x_{ij}(g)$ et les $x_i(q)$ sont des méridiens pour certaines variétés. Donc, pour montrer la proposition, il suffit de montrer chacune des relations pour un $\underline{q}_n$ arbitraire (voir corollaire \ref{conjMeridian}). \\
Supposons que les composantes de $\underline{q}_n$ sont concentrées dans un disque ne contenant pas deux éléments d'une même orbite pour l'action de $H$. Vu l'hypothèse imposée à $\underline{q}_n$, les $x_{ij}(1)$ vérifient à conjugaison près les relations de tresses pures (voir \cite{Art} pour la définition des relations de tresses pures), c'est à dire :
\begin{align*}
 (x_{ij}(1),x_{kl}(1))^\sim=(x_{il}(1), x_{jk}(1))^\sim=(x_{ik}(1), x_{jl}(1))^\sim=1 ,
\end{align*}
\begin{align*}
(x_{ij}(1),x_{ik}(1) \vert x_{jk}(1))^\sim&=( x_{jk}(1),x_{ij}(1)\vert x_{ik}(1))^\sim\endline
&=(x_{ik}(1), x_{jk}(1)\vert x_{ij}(1))^\sim=1,
\end{align*}
dans $\pi_1(\mathrm{Cf}_H(n,S_*),\underline{q}_n)$. Pour montrer (\ref{2indices}) et (\ref{3indices}), il suffit de calculer à l'aide de la proposition \ref{G.Lacet} l'image du premier commutateur par $f_\sigma(g_ih_k)$, l'image du deuxième et du troisième par $f_\sigma(g_ih_j)$ et celle des trois derniers par $f_\sigma(g_ih_k^{-1})$.\\\\
On va montrer la relation $(x_{i}(p), x_{jk}(g) )^\sim=1$ de (\ref{3indicesp}), les deux autres relations de la même équation étant similaires. On commence par prendre un petit disque $D$ autour de $p$, puis on choisit $\underline{q}_n$ tel que $q_i$ soit dans $D$ et les $q_r$ soient à l'extérieur de $H\cdot D$ ($r\neq i$). Ensuite, on choisit un $x'_{jk}(g)$ (conjugué de $x_{jk}(g)$) dont les $r$-ième composantes ($r\neq i$) ne coupent pas $H\cdot D$, puis un conjugué $x'_i(p)$ de $x_i(p)$ dont la $i$-ème composante est à image dans $D\setminus \{p\}$. Les lacets $x'_i(p)$ et $x'_{jk}(g)$ ainsi obtenus commutent.\\\\
Soit $p$ et $q$ dans  $S\setminus S_*$ avec $q\notin \mathcal{O}(p)$. Compte tenu de l'hypothèse sur $p$ et $q$, on peut trouver deux disques $D_p$ et $D_q$ centrés respectivement en $p$ et $q$ tels que $D_p \cap (H\cdot D_q)=\emptyset$. Choisissons $\underline{q}_n$ de façon à ce que $q_i$ soit dans $D_p$, $q_k$ soit dans $D_q$ et les autres composantes soient à l'exterieur de $H\cdot (D_p\cup D_q)$. Enfin, prenons des $\gamma_i(p)$, $\gamma_k(q)$ (les lacets qui permettent de définir $x_i(p)$ et $x_k(q)$), tels que $\mathrm{Im}(\gamma_i(p))\subset D_p$ et $\mathrm{Im}(\gamma_k(q))\subset D_q $. Les lacets $x_i(p)$ et $x_k(q)$ ainsi choisis commutent. Ce qui prouve (\ref{2ind-orb}).\\\\ 
Démontrons (\ref{stab1}) et (\ref{stab2}). Soit $D_p$ un disque autour de $p$ comme dans la proposition \ref{disc,stab}. Notons $w : \C\overset{\simeq}{\longrightarrow}  D_p$ et $w' : \mu_N \overset{\simeq}{\longrightarrow} \mathrm{stab}(p)$ les identifications qui établissent l'équivalence entre l'action de $\mu_N$ sur $\C$ et celle de $\mathrm{stab}(p)$ sur $D_p$. Choisissons un $\underline{q}_n$ tel que $q_i=w(2)$, $q_j=w(1)$ et $q_r \notin \underset{h\in H}{\cup} h\cdot D_p$, pour $r\neq i,j$. En utilisant (2) de la proposition \ref{disc,stab}, on voit qu'on a une inclusion :
\begin{align*}
w_{ij}  : \mathrm{Cf}_{\mu_N}(2,\C) & \to \mathrm{Cf}_H(n,S_*) \\
(x,y) & \mapsto (q_1,\cdots,q_{i-1},w(x),q_{i+1},\cdots ,q_{j-1},w(y),q_{j+1},\cdots,q_n).
\end{align*}
Remarquons que $w_{ij}\circ x_{12}(\zeta)$, pour $\zeta \in \mu_N$ est un méridien pour $(i,j,w'(\zeta))$. Donc, (1) et (2) du lemme \ref{DT} donnent après application de $w_{ij}$ :
$$(x_j(p) \vert x_{ij}(1) \vert x_{ij}(w'(e^{\frac{2i\pi}{N}}))\vert  \cdots \vert x_{ij}(w'(e^{\frac{2i\pi(N-1)}{N}}))\vert x_i(p) ,x_{ij}(1))^\sim=1,$$
$$(x_{ij}(1) \vert x_{ij}(w'(e^{\frac{2i\pi}{N}}))\vert  \cdots \vert x_{ij}(w'(e^{\frac{2i\pi(N-1)}{N}}))\vert x_{i}(p),x_{j}(p))^\sim=1.$$
Ce qui montre (\ref{stab1}) et (\ref{stab2}) pour $g=1$. Pour obtenir (\ref{stab1}) et (\ref{stab2}) pour $g$ quelconque, il suffit de calculer l'image par $f_\sigma(g_i)$ des deux dernières relations obtenues.\\\\
Pour montrer la dernière relation, observons l'ensemble : 
 \[\mathrm{Fib}^{(i)}=\mathrm{Cf}_H(n,S_*^2) \cap  (\{q_1\}\times \cdots \times\{ q_{i-1}\} \times S^2\times \{q_{i+1}\}\times \cdots \times\{ q_{n}\} ).\]
C'est  $S^2_*$ privée de l'union $E_i:=\cup_{k\neq i } (H\cdot q_k)$. Donc, un lacet contractile de $\mathrm{Fib}^{(i)}$ s'écrit comme un produit des lacets $\alpha_i(p)$, pour $p$ parcourant $E_i$ et $S^2\setminus S^2_*$, $\alpha_i(p)$ est basé en $q_i$ et fait un tour autour de $p$ en évitant les autres points de $E_i$. Soit $p$ de la forme $p=h\cdot q_j$ pour un certain $h\in H$ et un $j\in[1,n]\setminus \{i\} $. Si $j>i$, le lacet $\alpha _i(p)$ n'est qu'un conjugué de $x_{ij}(h)$. Sinon, $\alpha_i(p)$ est un méridien pour $(j,i,h^{-1})$ et donc il est conjugué à $x_{ji}(h^{-1})$. Ce qui montre la dernière relation.
\end{proof}
\section{ Rappels sur l'algèbre de Lie de Malcev d'un groupe}
On rappelle des notions reliées à l'algèbre de Lie de Malcev d'un groupe. Ce matériel provient de \cite{Q1} et \cite{Q2}. \\\\
Soit $\K$ un corps de caractéristique nulle et $A$ une algèbre de Hopf complète sur $\K$. Par définition, $A$ est munie d'une filtration décroissante $\{F_kA\}_{k\geq 0}$ multiplicative de sous-espaces vectoriels, telle que $F_0A=A$, $F_1A$ est l'idéal d'augmentation et la diagonale $\Delta : A \to A\widehat{\otimes}A $ est un morphisme d'algèbres filtrés. Cette filtration induit sur l'ensemble des éléments primitifs de $A$ une filtration d'algèbre de Lie :
\begin{equation}\label{filH} F_{k} \mathcal{P}(A)=\mathcal{P}(A)\cap F_k A, \: \text{pour} \: k\geq 1, \end{equation}
c'est-à-dire $[F_m \mathcal{P}(A),F_l \mathcal{P}(A)] \subset F_{m+l} \mathcal{P}(A)$, pour $m,l\geq1$. En particulier, le $k$-ième terme de la filtration centrale descendante de $\mathcal{P}(A)$ est inclus dans $F_k\mathcal{P}(A)$. \\
De plus, on a les bijections inverses $  \mathcal{G}(A) \underset{\mathrm{exp}}{\overset{\mathrm{log}}{\rightleftarrows}} \mathcal{P}(A)$, où $\mathrm{log}$ et $\mathrm{exp}$ sont les séries usuelles. La multiplication sur $\mathcal{G}(A)$ est donnée par : $$\mathrm{exp}(x)\mathrm{exp}(y)=\mathrm{exp}(\mathfrak{h}(x,y)), \: \text{pour}\: x,y \in \mathcal{P}(A) $$
où $\mathfrak{h}(x,y)=x+y +\frac{1}{2} [x,y] + \mathfrak{h}_{\geq 3}(x,y)$ avec $\mathfrak{h}_{\geq 3}(x,y) \in F_3 \mathcal{P}(A)$ est la série de Baker-Campbell-Hausdorff.\\\\
On considérera les complétions de deux types d'algèbres de Hopf relativement aux puissances de leur idéal d'augmentation :  l'algèbre $\K [\Gamma]$ d'un groupe $\Gamma$ ou l'algèbre enveloppante $\mathrm{U} (\mathfrak{g})$ d'une l'algèbre de Lie $\mathfrak{g}$. On notera $\K [\Gamma]^{\widehat{ \: \:}}$ et ${\widehat{\mathrm{U}}}(\mathfrak{g})$ ces complétés.\\\\
La $\K$-algèbre de Lie de Malcev du groupe $\Gamma$ est l'algèbre de Lie $\mathrm{Lie}(\Gamma(\K))$ formée des primitifs de $\K [\Gamma]^{\widehat{ \: \:}}$, munie de la filtration $\{F_k \mathrm{Lie}(\Gamma(\K))\}_{k\geq 1} $ comme dans (\ref{filH}). C'est une algèbre de Lie complète : 
$$\mathrm{Lie}(\Gamma(\K))=\underset{\longleftarrow}{\mathrm{lim}} \ \mathrm{Lie}(\Gamma(\K))/ F_k \mathrm{Lie}(\Gamma(\K)).$$

On note $\mathrm{gr}\mathrm{Lie}(\Gamma(\K))$ le gradué associé de $\mathrm{Lie}(\Gamma(\K))$ pour la filtration $\{F_k \mathrm{Lie}(\Gamma(\K))\}_{k\geq 1} $.  L'algèbre de Lie :  
\begin{equation}\label{dec}
\mathrm{gr} \mathrm{Lie}(\Gamma(\K))= {\bigoplus}_{k>0} \: \mathrm{gr}_k \mathrm{Lie}(\Gamma(\K))\end{equation} 
est engendrée par sa composante de degré un $ \mathrm{gr}_1 \mathrm{Lie}(\Gamma(\K))$. Soit $I$ l'idéal d'augmentation de $\K[\Gamma]$. On a la suite d'isomorphismes :
\begin{equation}\label{Deg1}
\mathrm{gr}_1 \mathrm{Lie}(\Gamma(\K)) \overset{\sim}{\longrightarrow} I/I^2 \overset{\sim}{\longrightarrow} \Gamma^{ab}\otimes_{\Z} \K
\end{equation}
qui pour $x\in \Gamma$, identifient la classe de $\mathrm{log}(x)$ dans $\mathrm{gr}_1 \mathrm{Lie}(\widehat{\Gamma}(\K))$ à $[x]\otimes 1$ dans $\Gamma^{ab}\otimes_{\Z} \K$ où $[x]$ est la classe de $x$ dans l'abélianisé $\Gamma^{ab}$  de $\Gamma$.
     
\section{L'isomorphisme entre $\mathrm{Lie}(\Gamma_n(\C))$, $\widehat{\mathrm{gr}}\mathrm{Lie}(\Gamma_n(\C))$ et $\widehat{\mathfrak{p}}_n(G)(\C)$}
Dans cette section on établit l'isomorphsime entre les trois algèbres de Lie $\mathrm{Lie}(\Gamma_n(\C))$, $\widehat{\mathrm{gr}}\mathrm{Lie}(\Gamma_n(\C))$ et $\widehat{\mathfrak{p}}_n(G)(\C)$. En s'appuyant sur les relations dans  $\Gamma_n:=\pi_1(\mathrm{Cf}_G(n,\PP^1_*))$ obtenues dans la section 3, on construit (sous-section 5.1) un morphisme  surjectif $\phi_{\K}$ de $\widehat{\mathfrak{p}}_n(G)(\K) $ dans $\widehat{\mathrm{gr}}\mathrm{Lie}(\Gamma_n(\K))$, pour $\K$ un corps de caractéristique nulle (ces objets sont définis dans la sous-section 1.2 et la section 4). Dans la sous-section 5.2, on utilise la représentation de monodromie de la section 1 pour obtenir un morphisme de Lie surjectif  $\mathrm{Lie(\rho)} : \mathrm{Lie}(\Gamma_n(\C)) \to \widehat{\mathfrak{p}}_n(G)(\C)$. Enfin, dans la sous-section 5.3, on utilise ces deux morphismes pour montrer l'isomorphisme annoncé.\\\\ 
Pour alléger les notations, on utilisera parfois (dans les démonstrations) $\mathfrak{p}_n$ pour désigner $\mathfrak{p}_n(G)(\K)$ et on omettera $\K$ dans $\mathrm{Lie}(\Gamma_n(\K))$. La classe de $x\in  F_k \mathrm{Lie}(\Gamma_n(\K))$ dans $\mathrm{gr}\mathrm{Lie}(\Gamma_n(\K))$ sera notée $[x]_k$. On notera $y$ l'image de $y \in \Gamma$ dans $\K [\Gamma]^{\widehat{ \: \:}}$.

\subsection{ Construction d'un morphisme $\widehat{\mathfrak{p}}_n(G)(\K) \to \widehat{\mathrm{gr}}\mathrm{Lie}(\Gamma_n(\K))$} Dans cette sous-section on construit le morphisme annoncé, pour $\K$ un corps de caractéristique nulle. On commence par démontrer que $\mathfrak{p}_n(G)$ admet une variante de la présentation de la définition \ref{def AL} : 
\begin{lemma}\label{presAL}L'algèbre de Lie $\mathfrak{t}_n(G)$ engendrée par les $X^{ij}(g)$ et $X^k(q)$, pour $1\leq i<j\leq n $, $g\in G$, $k\in [1,n]$ et $q\in \PP^1\setminus\PP^1_*$, soumis au relations : 
\begin{equation}\label{1}
\underset{q\in \PP^1\setminus \PP^1_*}{\sum} X^i(q) +\underset{g\in G}{\sum}  \: (\underset{j \vert j>i}{\sum}X^{ij}(g)+\underset{j\vert j<i}{\sum}X^{ji}(g) \: )=0,\end{equation}
pour $i\in[1,n]$, 
\begin{equation}\label{2}
 [X^{ij}(g),X^{kl}(h)]=[X^{il}(g), X^{jk}(h)]=[X^{ik}(g), X^{jl}(h)]=0 ,
\end{equation}
pour $1\leq i<j<k<l \leq n$ et $g,h \in G$,
\begin{align}\label{3}
[X^{ij}(g),X^{ik}(gh) + X^{jk}(h)]&=[ X^{jk}(h),X^{ij}(g)+X^{ik}(gh)]\endline
&=[X^{ik}(gh), X^{jk}(h)+ X^{ij}(g)]=0,\endline
[X^{i}(p),X^{jk}(g)]&=[X^{j}(p),X^{ik}(g)]=[X^{k}(p),X^{ij}(g)]=0,
\end{align}
pour $1\leq i<j<k \leq n$ et $g,h \in G$,

\begin{equation}  \label{5}[X^{i}(p),X^{j}(q)]=0,\end{equation}
 \begin{equation}\label{6}[X^{ij}(g),X^j(p) +X^i(g\cdot p) +\: \underset{h\in \mathrm{stab}(p)}{\sum} X^{ij}(gh)]=0,\end{equation} 
\begin{equation}\label{7} [X^{j}(p),X^i(g\cdot p) +\: \underset{h\in \mathrm{stab}(p)}{\sum} X^{ij}(gh)]=0,\end{equation}

pour $i,j \in [1,n]$ avec $i<j$, $g\in G$, $p \in \PP^1\setminus \PP^1_*$ et $q\in (\PP^1\setminus \PP^1_*)\setminus \mathcal{O}(p)$,\\\\
est isomorphe à $\mathfrak{p}_n(G)$. L'isomorphisme est donné par $X^{ij}(g)\mapsto X_{ij}(g)$ $(\text{pour $i<j$ et $g\in G$})$ et $X^{k}(q)\mapsto X_{k}(q)$ $(\text{pour $k\in[1,n]$ et $q\in \PP^1\setminus \PP^1_*$})$. 
\end{lemma}
\begin{proof}
On adjoint à $\mathfrak{t}_n(G)$ les éléments $X^{ji}(g)$ pour $1\leq i<j \leq n$ et $g\in G$ et on impose la relation :
 \begin{equation}\label{adjo} 
X^{ij}(g)=X^{ji}(g^{-1}).\end{equation}
Cela ne change pas $\mathfrak{t}_n(G)$. En utilisant la nouvelle relation $X^{ij}(g)=X^{ji}(g^{-1})$, on vérifie qu'on peut transfomer  les relations (\ref{1}) jusqu'a (\ref{6}) de ce lemme en les relations (2) jusqu'a (6) de la définition \ref{def AL} via l'identification $X^{ab}(g)\mapsto X_{ab}(g), X^c(q)\mapsto X_c(q)$ $(\text{pour }a,b,c\in [1,n], \text{ avec } a\neq b)$.\\
On s'intéresse à la relation (\ref{7}) (satisfaite pour $i<j$) qui est similaire à (7) (satisfaite pour $i\neq j$). L'antisymétrie du crochet donne : 
$$ [X^j(p), X] +[X^i(g\cdot p) ,X] + \underset{h'\in \mathrm{stab}(p)}{\sum} [X^{ij}(gh'),X]=0,$$ 
pour $X=X^j(p) +X^i(g\cdot p) +\: \underset{h\in \mathrm{stab}(p)}{\sum} X^{ij}(gh)$. Or, le crochet $[X^j(p), X]$ est nul d'après (\ref{7}) et $[X^{ij}(gh'),X]$ est nul d'après (\ref{6}). D'où : $$ [X^i(g\cdot p), X^j(p) +X^i(g\cdot p) +\: \underset{h\in \mathrm{stab}(p)}{\sum} X^{ij}(gh)]=0.$$
En appliquant $ X^{ij}(g)=X^{ji}(g^{-1})$ et en réindexant, on trouve :
$$ [X^i(q), X^j(g^{-1}\cdot q) +\: \underset{h\in \mathrm{stab}(q)}{\sum} X^{ji}(g^{-1}h')]=0 $$ 
où $q=g\cdot p$. Donc la relation (7) est aussi satisfaite dans $\mathfrak{t}_n(G)$ (via l'identification précédente). Ce qui montre le lemme.
\end{proof}

\begin{proposition}\label{pres gr}
On a un morphisme d'algèbres de Lie filtrées :
\[\phi_{\K} : \widehat{\mathfrak{p}}_n(G)(\K) \longrightarrow \widehat{\mathrm{gr}}\mathrm{Lie}(\Gamma_n(\K)), \]
donné par $\phi(X_{ij}(g))=[\mathrm{log}(x_{ij}(g))]_1$, pour $1 \leq  i<j \leq n$ et $g\in G$, où les complétions et les filtrations sont induites par le degré.
\end{proposition}
\begin{proof}
Posons $X^{ij}(g)=[\mathrm{log}(x_{ij}(g))]_1$ et $X^i(p)=[\mathrm{log}(x_{i}(p))]_1$. La formule de Campbell-Baker-Hausdorff montre que pour $x,y\in \mathcal{G}(\K [\Gamma_n]^{\widehat{ \: \:}})$ et  $x',y'$ des conjugués respectifs de $x$ et $y$ dans $\mathcal{G}(\K [\Gamma_n]^{\widehat{ \: \:}})$, on a :
\begin{equation} [\: [\mathrm{log}(x)]_1,[\mathrm{log}(y)]_1 \:]\in[\mathrm{log}(x',y')]_2 + F_3 \mathrm{Lie}(\Gamma_n(\K)).\end{equation}
En effet, $F_i \mathrm{Lie}(\Gamma_n(\K))$ contient le $i$-ème terme de la suite centrale descendante de $\mathrm{Lie}(\Gamma_n(\K))$.
En appliquant cette égalité aux relations de la proposition \ref{Rel Quad}, on trouve que les $X^{ij}(g)$ et $X^k(q)$ ainsi définit satisfonts les relations du lemme \ref{presAL} $(\text{dans }\mathrm{gr}_2 \mathrm{Lie}(\Gamma_n))$. Ce qui donne le morphisme de la proposition. 
\end{proof}

\subsection{Un morphisme de $\mathrm{Lie}(\Gamma_n(\C))$ dans $ \widehat{\mathfrak{p}}_n(G)(\C) $ }
On va construire un morphisme surjectif d'algèbres de Lie filtrées $\mathrm{Lie}(\Gamma_n(\C))\longrightarrow \widehat{\mathfrak{p}}_n(G)(\C) $.

\begin{lemma}
On a un isomorphisme entre $\mathcal{P}(\widehat{\mathrm{U}}(\mathfrak{p}_n(G)(\K)))$ et $\widehat{\mathfrak{p}}_n(G)(\K)$ qui fait correspondre la filtration de (\ref{filH}) de $\mathcal{P}(\widehat{\mathrm{U}}(\mathfrak{p}_n(G)(\K)))$ à la filtration induite par le degré de $\widehat{\mathfrak{p}}_n(G)(\K)$.
\end{lemma}
\begin{proposition}
On a un morphisme surjectif d'algèbres de Lie filtrées $\mathrm{Lie(\rho)} : \mathrm{Lie}(\Gamma_n(\C)) \to \widehat{\mathfrak{p}}_n(\C)\:$, vérifiant pour $i< j$ et $g\in G$, $k\in [1,n]$ et $q\in \PP^1\setminus \PP^1_*$ : 
$$\mathrm{Lie(\rho)}(\mathrm{log}(x_{ij}(g)))=X_{ij}(g)) + R_{ij}(g) , \qquad  \mathrm{Lie(\rho)}(\mathrm{log}(x_{k}(g)))= X_k(q) + R_k(q), $$ 
où $R_{ij}(g)$ et $R_k(q)$ appartiennent à $F_2 \mathrm{Lie}(\Gamma_n(\C))$.
\end{proposition}
\begin{proof}
Dans la section 1.6 on a construit un anti-morphisme $\rho_{\tilde{\underline{q}_n}} : \Gamma_n \longrightarrow \mathcal{G}(\widehat{\mathrm{U}}\mathfrak{p}_n)$. On définit le morphisme $ \tilde{\rho}$ opposé de $\rho_{\tilde{\underline{q}}_n}$, par $ \tilde{\rho}(x)= \rho_{\tilde{\underline{q}}_n}(x)^{-1}$. Le morphisme ainsi obtenu s'étend en un morphisme compatible avec les structures de Hopf $ f: \C[\Gamma_n] \longrightarrow \widehat{\mathrm{U}}\mathfrak{p}_n$. En considérant la complétion de $\C[\Gamma_n]$, on obtient un morphisme $\hat{f}:\C [\Gamma_n]^{\widehat{ \: \:}}\longrightarrow \widehat{\mathrm{U}}\mathfrak{p}_n$ d'algèbres de Hopf complètes. En restreignant $\hat{f}$ aux primitifs, on trouve le morphisme d'algèbres de Lie filtrées $\mathrm{Lie}(\tilde{\rho}) : \mathrm{Lie}(\Gamma_n(\C)) \to \mathcal{P}(\widehat{\mathrm{U}}\mathfrak{p}_n)$, satisfaisant :
$$\mathrm{log}(x) \mapsto \mathrm{log}(\tilde{\rho}(x)),$$
 pour $x\in \Gamma_n$. Comme les lacets $x_{ij}(g)$ et $x_k(q)$ sont similaires aux $x_{ij}^g$ et aux $x_{i}^q$ de la sous-section 1.6, on obtient en utilisant la proposition \ref{monodromie} : 
$$\mathrm{Lie}(\tilde{\rho})(\mathrm{log}(x_{ij}(g)))=2\mathrm{i}\pi X_{ij}(g)+[deg \geq 2],\qquad \mathrm{Lie}(\tilde{\rho})(\mathrm{log}(x_{k}(q)))=2\mathrm{i}\pi X_{k}(q)+[deg \geq 2].$$ 
Enfin, on a un isomorphisme canonique d'algèbres de Lie filtrées entre $\mathcal{P}(\widehat{\mathrm{U}}\mathfrak{p}_n) $ et $\widehat{\mathfrak{p}}_n(\C)\: $ ; le morphisme $\mathrm{Lie}(\rho)$ annoncé est obtenu en composant $\mathrm{Lie}(\tilde{\rho})$ par l'automorphisme de $  \widehat {\mathfrak{p}}_n(\C)$ donné par $2\mathrm{i}\pi X_{ij}(g) \mapsto X_{ij}(g), 2\mathrm{i}\pi X_{k}(q) \mapsto X_k(q)$.\\\\
Montrons la surjectivité de $\mathrm{Lie}(\rho)$. L'espace vectoriel $\mathrm{gr}_1  \widehat {\mathfrak{p}}_n=\mathfrak{p}_n^1$ est engendré par les $X_{ij}(g)$ ($i<j$), qui appartiennent à l'image de $\mathrm{gr} \mathrm{Lie} (\rho)$. En effet, $\mathrm{gr} \mathrm{Lie} (\rho)([\mathrm{log}(x_{ij}(g))]_1)=X_{ij}(g)$. Ainsi, $\mathrm{gr} \mathrm{Lie} (\rho)$ est surjectif car  $\mathfrak{p}_n$ est engendrée en degré un. Par conséquent, $\rm{Lie} (\rho)$ est surjectif.
\end{proof}
\subsection{L'isomorphisme sur $\C$}
\begin{lemma}\label{Gengr1}
L'espace $\mathrm{gr}_1 \mathrm{Lie}(\Gamma_n(\K))$ est engendré par les classes des $\mathrm{log}(x_{ij}(g))$ (pour $1\leq i<j\leq n$ et $g\in G$) et celles des $\mathrm{log}(x_{k}(q))$ pour $k \in [1,n]$ et $q\in \PP^1\setminus \PP^1_*$.
\end{lemma}
\begin{proof}
Il suffit d'appliquer (\ref{Deg1}) à $\Gamma_n$ en tenant compte de la proposition \ref{Gen}. 
\end{proof}
\begin{proposition}\label{surC}
Les morphismes $\phi_\K:\widehat{\mathfrak{p}}_n(G)(\K) \to \widehat{\mathrm{gr}}\mathrm{Lie}(\Gamma_n(\K)) $ et $\mathrm{Lie}(\rho) : \mathrm{Lie}(\Gamma_n(\C))\to \widehat{\mathfrak{p}}_n(G)(\C) $ sont des isomorphismes d'algèbres de Lie filtrées (où $\K$ est un corps de caractéristique nulle).
\end{proposition} 

\begin{proof}
Posons $\theta=\phi_{\C} \circ \mathrm{Lie}(\rho)$. On a les égalités : $\mathrm{gr} \theta( [\mathrm{log}(x_{ij}(g)]_1)=[\mathrm{log}(x_{ij}(g))]_1$ et $\mathrm{gr} \theta( [\mathrm{log}(x_{i}(q)]_1)=[\mathrm{log}(x_{i}(q))]_1$. Donc, $\mathrm{gr}\theta$ est l'identité d'après le lemme \ref{Gengr1}. 
On en déduit que $\mathrm{gr} \mathrm{Lie}(\rho)$ est un isomorphisme (on a vu que $\mathrm{gr} \mathrm{Lie}(\rho)$ est surjectif) et donc $\mathrm{gr} \phi_{\C}$ est aussi un isomorphisme. Compte tenu de la constuction de $\phi_{\C}$, on en déduit que l'application $\mathrm{gr}\phi_{\Q}$ est un isomorphisme ce qui implique que $\mathrm{gr}\phi_{\K}$ est un isomorphisme pour tout $\K$ de caractéristique nulle. Ce qui prouve que les isomorphismes de la proposition sont des isomorphismes d'algèbres de Lie filtrées.
\end{proof}

\section{ L'isomorphisme entre $\mathrm{Lie}(\Gamma_n(\Q))$, $\widehat{\mathrm{gr}}\mathrm{Lie}(\Gamma_n(\Q))$ et $\widehat{\mathfrak{p}}_n(G)(\Q)$}
Dans cette section, on montre que les algèbres de Lie $\mathrm{Lie}(\Gamma_n(\Q))$, $\widehat{\mathrm{gr}}\mathrm{Lie}(\Gamma_n(\Q))$ et $\widehat{\mathfrak{p}}_n(G)(\Q)$ sont isomorphes. Pour cela, on construit un schéma $\underline{\mathrm{Iso}}_1(\mathfrak{g},\mathfrak{h})$ et un schéma en groupes $\underline{\mathrm{Aut}}_1(\mathfrak{g})$ (sous-section 6.1) associés à des algèbres de Lie $\mathfrak{g}$ et $\mathfrak{h}$, et on montre que $\underline{\mathrm{Iso}}_1(\mathfrak{g},\mathfrak{h})$ est un torseur sous $\underline{\mathrm{Aut}}_1(\mathfrak{g})$ (sous-section 6.2). On montre que l'isomorphisme construit en proposition \ref{surC} est un point complexe de $\underline{\mathrm{Iso}}_1(\mathfrak{g},\mathfrak{h})$. On rappelle un résultat permettant d'établir l'existence de points rationnels dans certains torseurs (sous-section 6.2). On en déduit l'existence de l'isomorphisme annoncé (sous-section 6.3) .\\
\subsection{Les schémas $\underline{\mathrm{Iso}}_1(\mathfrak{g},\mathfrak{h})$ et $\underline{\mathrm{Aut}}_1(\mathfrak{g})$}
Dans cette sous-section, on construit des schémas conduisant à la définition d'un système projectif de torseurs.
On considère deux $\Q$-algèbres de Lie $\mathfrak{g}$ et $\mathfrak{h}$ filtrées complètes de filtrations respectives $\{F_i\mathfrak{g}\}_{i\geq 1} $ et $\{F_i\mathfrak{h} \}_{i\geq 1}$. On suppose que $\mathrm{gr}\: \mathfrak{g}$ et $\mathrm{gr}\: \mathfrak{h}$ sont engendrées par leurs composantes de degré un, que ces composantes sont de dimension finie et qu'on a un isomorphisme fixé $\psi : \mathfrak{g}/F_2 \mathfrak{g} \longrightarrow \mathfrak{h}/F_2 \mathfrak{h}$ entre ces composantes de degré un. 
On pose $\mathfrak{g}_i=\mathfrak{g}/F_i\mathfrak{g}$ et $\mathfrak{h}_i=\mathfrak{h}/F_i\mathfrak{h}$. Ainsi, on a :  $ \mathfrak{g}=\underset{\longleftarrow}{\mathrm{lim}} \: \mathfrak{g}_i \quad \text{et} \quad \mathfrak{h}=\underset{\longleftarrow}{\mathrm{lim}}\ \mathfrak{h}_i $. 
Pour $\mathfrak{K}$ une algèbre de Lie  et $k\geq 1$, on note $ \mathcal{C}_k \mathfrak{K}$ le $k$-ième terme de la suite centrale descendante :  $ \mathcal{C}_1 \mathfrak{K}= \mathfrak{K}$ et $ \mathcal{C}_{k+1} \mathfrak{K}=[ \mathcal{C}_k \mathfrak{K},\mathfrak{K}]$.\\
Enfin, un $\Q$-schéma affine est vu comme un foncteur représentable des $\Q$-anneaux dans les ensembles et on  définit un $\Q$-schéma en groupes pro-unipotent comme la limite inverse d'un système projectif de $\Q$-schémas en groupes algébriques unipotents. On utilisera parfois $\Q$-groupe au lieu de $ \Q$-schéma en groupes.

\subsubsection{Le schéma $\underline{\mathrm{Iso}}_1(\mathfrak{g},\mathfrak{h})$} On va définir un système projectif de $\Q$-schémas affines ; on définira $\underline{\mathrm{Iso}}_1(\mathfrak{g},\mathfrak{h})$ comme la limite inverse de ce système projectif.
\begin{lemma}\label{acorriger}
L'image de $F_r\mathfrak{g}$ dans $\mathfrak{g}_i$ est $\mathcal{C}_r\mathfrak{g}_i$, le $r$-ième terme de la suite centrale descendante de $\mathfrak{g}_i$.
\end{lemma}
\begin{proof}
Notons $F_r \mathfrak{g}_i$ l'image de $F_r \mathfrak{g}$ dans $ \mathfrak{g}_i$ et montrons que $\mathcal{C}_r\mathfrak{g}= F_r\mathfrak{g}_i$. Comme la filtration $F_r \mathfrak{g}$ est décroissante et que $\mathcal{C}_i\mathfrak{g}\subset F_i\mathfrak{g}$ , il suit de la définition de $\mathfrak{g}_i$ que $\mathcal{C}_r\mathfrak{g}_i=F_r\mathfrak{g}_i=0$, pour $r\geq i $. 
Ce qui montre la proposition pour $r\geq i$. Le fait que $\mathrm{gr}\:\mathfrak{g}$ soit engendrée en degré un est équivalent à dire que $\mathcal{C}_r\mathfrak{g}+F_{r+1}\mathfrak{g}=F_r\mathfrak{g}$ pour $r\geq 1$. 
Donc, par projection on a $\mathcal{C}_r\mathfrak{g}_i+F_{r+1}\mathfrak{g}_i=F_r\mathfrak{g}_i$. Cette dernière égalité pour $r=i-1$ donne $\mathcal{C}_{i-1}\mathfrak{g}_i=\mathcal{C}_{i-1}\mathfrak{g}_i+\mathcal{C}_{i}\mathfrak{g}_i=F_{i-1}\mathfrak{g}_i$ car $F_i\mathfrak{g}_i=\mathcal{C}_{i}\mathfrak{g}_i$. Enfin, une récurrence descendante sur $r\leq i$ montre la proposition.
\end{proof}
Soit $\underline{\mathrm{Iso}}_1(\mathfrak{g}_i,\mathfrak{h}_i)$ le $\Q$-schéma algébrique qui à un $\Q$-anneau $R$ associe l'ensemble $\underline{\mathrm{Iso}}_1(\mathfrak{g}_i,\mathfrak{h}_i)(R)$ des isomorphismes de $\mathfrak{g}_i\otimes R$ dans $\mathfrak{h}_i\otimes R$ dont l'abélianisé s'identifie à $\psi\otimes id_R:\mathfrak{g}_2\otimes R\to \mathfrak{h}_2 \otimes R$ via les identifications canoniques $\mathfrak{g}_i ^{ab}\simeq \mathfrak{g}_2$ et $\mathfrak{h}_i^{ab} \simeq \mathfrak{h}_2$ (voir lemme \ref{acorriger}).

\begin{proposition}\label{ISO1}
Les schémas $\underline{\mathrm{Iso}}_1(\mathfrak{g}_i,\mathfrak{h}_i)$ forment naturellement un système projectif :
$$ \cdots \longrightarrow \underline{\mathrm{Iso}}_1(\mathfrak{g}_2,\mathfrak{h}_2)\longrightarrow \underline{\mathrm{Iso}}_1(\mathfrak{g}_1,\mathfrak{h}_1) ;$$
on notera $\underline{\mathrm{Iso}}_1(\mathfrak{g},\mathfrak{h})$ la limite de ce système. Pour $R$ un $\Q$-anneau, $\underline{\mathrm{Iso}}_1(\mathfrak{g},\mathfrak{h})(R)$ est l'ensemble des isomorphismes d'algèbres de Lie filtrées  $\mathfrak{g}\widehat{\otimes}R \longrightarrow \mathfrak{h}\widehat{\otimes}R $ qui induisent l'isomorphisme $\psi \otimes id_R$.
\end{proposition}
\begin{proof} L'image des idéaux $F_k \mathfrak{g}\subset \mathfrak{g} $ (respectivement $F_k \mathfrak{h}\subset \mathfrak{h}$) dans $\mathfrak{g}_i$ (respectivement $\mathfrak{h}_i$) étant les idéaux caractéristiques $\mathcal{C}_k\mathfrak{g}_i$ (lemme \ref{acorriger}), les $\underline{\mathrm{Iso}}_1(\mathfrak{g}_i,\mathfrak{h}_i)$ forment un système projectif comme annoncé.\\
Enfin, un isomorphisme filtré $f$ de $\mathfrak{g}\widehat{\otimes} R $ dans $\mathfrak{h}\widehat{\otimes} R$ est entièrement déterminé par la donnée d'isomorphismes $f_j$ faisant commuter le diagramme : 
$$ \xymatrix{
    \cdots &\mathfrak{g}_i\otimes R \ar[l] \ar[d]_{f_i}  & \ar[l] \mathfrak{g}_{i+1}\otimes R \ar[d]_{f_{i+1}} & \mathfrak{g}_{i+2}\otimes R \ar[d]_{f_{i+2}} \ar[l]& \ar[l] \cdots \\
   \cdots  &\mathfrak{h}_i\otimes R \ar[l] & \ar[l] \mathfrak{h}_{i+1}\otimes R &  \ar[l] \mathfrak{h}_{i+2}\otimes R& \ar[l] \cdots 
  }$$
où les flèches horizontales sont les projections canoniques. La condition $"f$ induit $\psi\otimes id_R:\mathfrak{g}_2\otimes R\longrightarrow \mathfrak{h}_2 \otimes R"$ correspond à $"f_i$ induit $\psi\otimes id_R:\mathfrak{g}_2\otimes R\longrightarrow \mathfrak{h}_2 \otimes R$ via les identifications canoniques $\mathfrak{g}_i ^{ab}\simeq \mathfrak{g}_2$ et $\mathfrak{h}_i^{ab} \simeq \mathfrak{h}_2"$. Ce qui montre la dernière assertion de la proposition et achève la démonstration.
\end{proof}

\subsubsection{Le schéma $\underline{\mathrm{Aut}}_1(\mathfrak{g})$}
On va construire le $\Q$-schéma en groupes $\underline{\mathrm{Aut}}_1(\mathfrak{g})$.\\\\
Soit $\underline{\mathrm{Aut}}(\mathfrak{g}_i)$ le foncteur qui associe à un $\Q$-anneau $R$ le groupe des automorphismes d'algèbre de Lie de $\mathfrak{g}_i \otimes R$. C'est un $\Q$-groupe algébrique. En effet, $\underline{\mathrm{Aut}}(\mathfrak{g}_i)$ est un sous-foncteur en groupes de $\mathrm{GL}(\mathfrak{g}_i)$, représentable par un quotient de $\mathcal{O}(\mathrm{GL}(\mathfrak{g}_i))$, l'anneau de coordonnées de $\mathrm{GL}(\mathfrak{g}_i)$. On définit le $\Q$-groupe algébrique $\underline{\mathrm{Aut}}_1(\mathfrak{g}_i)$ comme étant le noyau du morphisme naturel $\underline{\mathrm{Aut}}(\mathfrak{g}_i)\to \mathrm{GL}(\mathfrak{g}_i^{ab}) $.

\begin{proposition}
Les $\underline{\mathrm{Aut}}_1(\mathfrak{g}_i)$ sont des $\Q$-groupes algébriques unipotents, ils forment naturellement un système projectif : 
$$ \cdots \longrightarrow \underline{\mathrm{Aut}}_1(\mathfrak{g}_2)\longrightarrow \underline{\mathrm{Aut}}_1(\mathfrak{g}_1), $$
de limite  $\underline{\mathrm{Aut}}_1(\mathfrak{g})$ pro-unipotente. Pour un anneau $\Q$-anneau $R$, $\underline{\mathrm{Aut}}_1(\mathfrak{g})(R)$ est l'ensemble des automorphismes d'algèbre de Lie filtrée de $\mathfrak{g}\widehat{\otimes} R$ induisant l'identité sur $\mathfrak{g}/ F_2\mathfrak{g}$.
\end{proposition}
\begin{proof}
Le monomorphisme naturel $\underline{\mathrm{Aut}}_1(\mathfrak{g}_i)\to \mathrm{GL}(\mathfrak{g}_i)$ fournit une représentation linéaire fidèle de dimension finie de $\underline{\mathrm{Aut}}_1(\mathfrak{g}_i)$, lequel est algébrique. Le drapeau $\mathcal{C}_{n_i}\mathfrak{g}_i\subset \mathcal{C}_{n_i-1} \mathfrak{g}_i\subset \cdots \subset \mathcal{C}_{0} \mathfrak{g}_i=\mathfrak{g}_i$ ($n_i $ est la classe de nilpotence de $\mathfrak{g}_i$) est stable par $\underline{\mathrm{Aut}}(\mathfrak{g}_i)$. Comme $\underline{\mathrm{Aut}}_1(\mathfrak{g}_i)$ est le noyau de $\underline{\mathrm{Aut}}(\mathfrak{g}_i)\longrightarrow \mathrm{GL}(\mathfrak{g}_i^{ab}) $, il agit trivialement sur $\mathcal{C}_{k} \mathfrak{g}_i/ \mathcal{C}_{k+1}\mathfrak{g}_i$ pour $1\leq k < n_i $. Ce qui montre que $\underline{\mathrm{Aut}}_1(\mathfrak{g}_i)$ est unipotent.\\ 
Enfin, remarquons que $\underline{\mathrm{Aut}}_1(\mathfrak{g})$ n'est autre que $\underline{\mathrm{Iso}}_1(\mathfrak{g},\mathfrak{h})$ pour $\mathfrak{h}=\mathfrak{g}$ et $\psi=id_{\mathfrak{g}_2}$. Donc, le reste de la proposition est une conséquence de la proposition \ref{ISO1} et de la définition d'un groupe pro-unipotent.
\end{proof}

\subsection{Torseurs et algèbres de Lie} Dans cette partie on rappelle la notion de torseur, on montre que $\underline{\mathrm{Iso}}_1(\mathfrak{g},\mathfrak{h})$ est un torseur sous $\underline{\mathrm{Aut}}_1(\mathfrak{g})$ et on montre qu'un tel torseur a des points rationnels dans certaines conditions.
\begin{definition}
Un $\Q$-torseur est un $\Q$-schéma $X$, muni d'une action à gauche d'un $\Q$-schéma en groupes $H$ telle que l'action de $H(\K)$ sur $X(\K)$ est libre et transitive quand $X(\K)$ est non-vide.  On dit que $X$ est un torseur sous $H.$
\end{definition}
Soit $ \cdots \longrightarrow X_2\longrightarrow X_1 $ un système projectif de $\Q$-schémas, $X=\underset{\longleftarrow}{\mathrm{lim}} \: X_i$ sa limite projective et $H = \underset{\longleftarrow}{\mathrm{lim}} \: H_i$ un $\Q$-schéma en groupes pro-unipotent.
\begin{proposition} [\cite{BE2}] \label{Points Torseur}
Supposons que les $X_i$ forment un système de torseurs compatible sous les $H_i$ et que $X(\C)$ est non vide, alors $X(\Q)$ est non vide.
\end{proposition}

\begin{proposition}\label{Torseur}
Chaque $\underline{\mathrm{Iso}}_1(\mathfrak{g}_i,\mathfrak{h}_i)$ est un torseur sous l'action de $\underline{\mathrm{Aut}}_1(\mathfrak{g}_i)$ ; leur limite inverse est $\underline{\mathrm{Iso}}_1(\mathfrak{g},\mathfrak{h})$ qui est un torseur sous l'action de $\underline{\mathrm{Aut}}_1(\mathfrak{g},\mathfrak{h})$
\end{proposition}
\begin{proof}
Immédiat.
\end{proof}
En combinant les propositions \ref{Points Torseur} et \ref{Torseur}, on obtient : 
\begin{corollary}\label{CQ}
Si $\underline{\mathrm{Iso}}_1(\mathfrak{g},\mathfrak{h})(\C)\neq \emptyset$ alors $\mathfrak{g}$ et $\mathfrak{h}$ sont isomorphes en tant qu'algèbres de Lie filtrées.
\end{corollary}
\subsection{Preuve} On va montrer : 
\begin{theorem}\label{Th}
Les algèbres de Lie $\mathrm{Lie}(\Gamma_n(\Q))$, $\widehat{\mathrm{gr}}\mathrm{Lie}(\Gamma_n(\Q))$ et $\widehat{\mathfrak{p}}_n(G)(\Q)$ sont isomorphes en tant qu'algèbres de Lie filtrées.
\end{theorem}
\begin{proof}
 Le quotient $ \mathrm{Lie}(\Gamma_n(\Q))/F_2 \mathrm{Lie}(\Gamma_n(\Q))$ s'identifie canoniquement à l'espace $ \mathrm{gr}_1 \mathrm{Lie}(\Gamma_n(\Q))= \widehat{\mathrm{gr}} \mathrm{Lie}(\Gamma_n(\Q)) /F_2 \:\widehat{\mathrm{gr}} \mathrm{Lie}(\Gamma_n(\Q))$. Notons $\psi_n$ cette identification. Considérons le $\Q$-schéma $X_{\Gamma_n}= \underline{\mathrm{Iso}}_1(\mathrm{Lie}(\Gamma_n(\Q)),\widehat{\mathrm{gr}} \mathrm{Lie}(\Gamma_n(\Q)))$ comme dans la section précédente pour $\psi=\psi_n$. 
L'isomorphisme $\theta=\phi_{\C} \circ \mathrm{Lie}(\rho)$ (de la proposition \ref{surC} et sa preuve) est un élément de $X_{\Gamma_n}(\C)$. 
En effet, on a  $\mathrm{Lie}(\Gamma_n(\C))=\mathrm{Lie}(\Gamma_n(\Q))\widehat{\otimes} \C$, $\widehat{\mathrm{gr}}\mathrm{Lie}(\Gamma_n(\C))=\mathrm{gr}\mathrm{Lie}(\Gamma_n(\Q))\widehat{\otimes} \C$ et on a vu que $\mathrm{gr}\theta$ est l'identité et donc il induit $\psi_n$. Par conséquent, $\mathrm{Lie}(\Gamma_n(\Q))$ et $\widehat{\mathrm{gr}}\mathrm{Lie}(\Gamma_n(\Q))$ sont isomorphes en tant qu'algèbres de Lie filtrées d'après le corollaire \ref{CQ}. Le fait que $\phi_\Q: \widehat{\mathfrak{p}}_n(G)(\Q) \to \widehat{\mathrm{gr}}\mathrm{Lie}(\Gamma_n(\Q))$ est un isomorphisme provient de la proposition \ref{surC}.
\end{proof}
\renewcommand{\abstractname}{\textbf{Remerciements}}
\begin{abstract}
Je tiens à remercier Benjamin Enriquez pour sa lecture critique et ces nombreux conseils.  
\end{abstract}
\bibliography{test}{}
\bibliographystyle{plain}

\end{document}